%% file: ex_article.tex
\documentclass[onefignum,onetabnum,reqno]{siamart190516}%review,

\input{ex_shared}

\begin{document}

\maketitle

% REQUIRED
\begin{abstract}
	This paper presents a systematic theoretical framework to 
	 derive the energy identities 
	 of {\em general implicit and explicit} Runge--Kutta (RK) methods for linear seminegative systems. 
	 It generalizes the stability analysis of {\em explicit} RK methods in [Z.~Sun and C.-W.~Shu, {\em SIAM J. Numer. Anal.}, 57 (2019), pp.~1158--1182]. 
		The established energy identities provide a precise characterization
	on whether and how the energy dissipates in the RK discretization, 
	thereby leading to  weak and strong stability criteria of 
	RK methods. Furthermore, we discover a unified energy identity for all the diagonal Pad\'e approximations,  
	based on an analytical Cholesky type decomposition of 
	a class of symmetric matrices. The structure of the matrices is very complicated, rendering 
	the discovery of the unified energy identity and the proof of the decomposition
	 highly challenging.  
	Our proofs involve the construction of technical 
	combinatorial identities and novel techniques from the theory of hypergeometric series. 
	Our framework is motivated by a discrete analogue of integration by parts technique 
	and a series expansion of the continuous energy law. 
	In some special cases, our analyses establish a close connection between  
	the continuous and discrete energy laws, enhancing our understanding of their intrinsic mechanisms. 
	Several specific examples of implicit methods are given to illustrate the discrete energy laws. A few numerical examples further confirm the theoretical properties. %\zsc{Should we use energy equalities or identities?}

\end{abstract}

%\zs{Things to do:
%\begin{enumerate}
%	\item Some times RK methods may cause energy production. They may not always have a *dissipation* law. Is there a better title for our paper?
%	\item Prove \cref{lem:uuhat-cid}.
%	\item Postpone technical proofs of \cref{sec:choleskyPade} to appendix.
%	\item Add numerical tests in \cref{sec:num}. 
%	\item Add abstract and conclusions. 
%\end{enumerate}}

% REQUIRED

% REQUIRED

\section{Introduction}
This paper is concerned with the autonomous linear seminegative differential systems in a general form: 
\begin{equation}\label{eq-odes}
	\dt u = L u, \quad u = u(t)\in L^2([0,T];V),
\end{equation}
where $V$ is a finite or infinite dimensional real Hilbert space equipped with the inner product $\ip{\cdot, \cdot}$ and the induced norm $\| \cdot\|$, and $L$ is a bounded linear seminegative operator satisfying $\ip{Lv,v} \leq 0$ for all $v\in V$. (The operator $L$ is not necessarily normal, namely, it may not commute with its adjoint.) 
A typical example of \cref{eq-odes} is the linear seminegative ordinary differential equations (ODEs) with  $V=\mathbb R^{N_d}$, 
$\ip{\cdot, \cdot}$ being the standard $l^2$ inner product, and the operator $L$ being a seminegative  $N_d\times N_d$ real constant matrix. 
Such ODEs may also arise from suitable semi-discrete schemes 
%, e.g., the method of lines, 
for some linear partial differential equations (PDEs), such as linear hyperbolic or convection-diffusion equations, etc. 
%On the other hand, system \cref{eq-odes} can also be PDEs or integro-differential equations, for which the linear operator $L$ may involve derivatives or integrals with respect to other independent variables except $t$. 
The seminegative operator $L$ induces a semi-inner-product $\qip{\cdot,\cdot }$ on $V$ defined by 
\begin{equation}\label{eq:Def-semi-inner}
	\qip{w,v} := -\ip{Lw,v}-\ip{w,Lv}.
\end{equation}
The corresponding semi-norm is denoted as $\qnm{v} := \sqrt{\qip{v,v}}$. Then it can be seen that the system \eqref{eq-odes} admits the following energy dissipation law
\begin{equation}\label{eq:WKL01}
	\dt \nm{u}^2 = \left \langle{\dt u, u} \right \rangle+ \left \langle {u, \dt u} \right \rangle = \ip {Lu, u} + \ip {u, Lu} = -\qnm{u}^2 \le 0.
\end{equation}
Furthermore, if we integrate \cref{eq:WKL01} in time from $t^n$ to $t^{n+1} := t^n +\tau$ with $\tau>0$, then it yields
\begin{equation}\label{eq:ode-ed}
\nm{u(t^{n+1})}^2 - \nm{u(t^{n})}^2 = -\int_{0}^\tau \qnm{u(t^n+\widehat{\tau})}^2\mathrm{d}\widehat{\tau} \le 0. 
\end{equation}

The Runge--Kutta (RK) methods are widely used in temporal discretization for the approximate solutions of ODEs and time-dependent PDEs. 
In this paper, we discretize system \eqref{eq-odes} with RK methods, and we wish to establish a systematic framework to study how the energy law \eqref{eq:ode-ed} is approximated in generic RK discretization. 
The discrete energy laws are important and helpful for further understanding the stability of RK methods,     
%In this paper, we discretize \eqref{eq-odes} with Runge--Kutta (RK) methods and study how the $L^2$ energy dissipation law \eqref{eq:ode-ed} is approximated. Our main focus will be on implicit RK methods. In particular, diagonal Pad\'e approximations will be systematically studied. 
which is a classical topic in numerical analysis. Over the past decades, rich mathematical theories on the stability of RK methods  have been developed both in the ODE settings (see \cite[Chapter IV]{wanner1996solving}, \cite[Chapter 3]{butcher2016numerical},  and references therein) and in the context of numerical PDEs (see \cite{burman2010explicit,zhang2010stability,deriaz2012stability,wang2015stability,xu20192,xu2020superconvergence,drake2021convergence} and references therein).

One classical way to analyze the stability of RK methods is through the eigenvalue analysis, which typically focuses on the scalar ODE $\frac{d}{dt} u = \lambda u$ with a complex constant $\lambda$. Specifically, an RK method applied to this scalar ODE reduces to the iteration $u^{n+1} = \mathcal{R}(\tau \lambda) u^n$ and the stability criterion is then imposed as $|\mathcal{R}(\tau \lambda)|\leq 1$. In the special case that the stability region (where $|\mathcal{R}(\tau \lambda)|\leq 1$ holds) contains the left complex plane, the methods are called A-stable \cite{dahlquist1963special}. It is noted that for an A-stable RK method, the unconditional stability for the scalar equation implies the $L^2$-stability for the linear seminegative system \eqref{eq-odes} in the sense that $\nm{u^{n+1}}\leq \nm{u^n}$. A proof of this implication was  given in \cite[Chapter IV. 11]{wanner1996solving} based on a lemma by von Neumann \cite{neumann1950spektraltheorie}; see also \cite{hairer1982stability}. However, for the RK methods that are not A-stable, special attention should be paid when extending the analysis from the scalar equation to the ODE system \eqref{eq-odes}. If the operator $L$ in \eqref{eq-odes} is normal, namely it commutes with its adjoint, then the system \eqref{eq-odes} can be unitarily diagonalized into decoupled scalar equations. In this case, the eigenvalue analysis will provide a both necessary and sufficient stability criterion. However, when $L$ is not normal, which is generic for the ODE system \eqref{eq-odes} obtained from semi-discrete PDE schemes, the 
eigenvalue analysis gives only necessary but possibly insufficient conditions for stability. 
This is due to the gap between the spectral radius and the operator norm. 
Therefore, the eigenvalue analysis may sometimes give misleading conclusions on the time step constraint \cite{iserles2009first} or the stability property \cite{sun2017rk4}.

To overcome the above-mentioned limitation, the energy method can be used as an alternative approach for stability analysis, which seeks certain energy identity or inequality. 
%An alternative approach to overcome the above-mentioned limitation is the energy method, which seeks   certain energy equality or inequality for stability analysis of the RK methods. 
%This approach  could overcome the limitation of eigenvalue analysis. 
 For implicit RK methods, their BN stability and algebraic stability \cite{burrage1979stability, wanner1996solving} were analyzed based on the energy method. For explicit RK methods, one stream of the research concerns the coercive operators \cite{levy1998semidiscrete}, which typically arise from diffusive problems such as method-of-lines schemes for the heat equation. It was shown that the Euler forward method is able to preserve the monotonous decay property $\nm{u^{n+1}}\leq \nm{u^n}$ under suitable time step constraint \cite{gottlieb2001strong}. We will refer to this property as strong stability (sometimes also termed as monotonicity or monotonicity-preserving property in the literature \cite{higueras2005monotonicity, ketcheson2019relaxation}). This stability property can be extended to all strong-stability-preserving RK methods \cite{gottlieb2001strong,ketcheson2011strong,BrestenGottliebGrantHiggsKetcheson2017,IsherwoodGrantGottlieb2018}, which are constructed as convex combinations of Euler forward steps. In particular, those RK methods reducing to truncated Taylor expansions are all of such convex combination forms 
 %admit {\color{blue}this structure [be more precise]} 
 and thus strongly stable \cite{gottlieb2001strong}. These arguments also coincide with the contractivity analysis under the so-called circle condition in \cite{spijker1983contractivity, kraaijevanger1991contractivity} and can be extended to nonlinear problems. However, such arguments may not be generally applied to noncoercive problems that commonly arise from semidiscrete schemes for wave type equations. A high-order energy expansion has to be carried out. 
 Motivated by the studies on the third-order \cite{tadmor2002semidiscrete} and the fourth-order \cite{sun2017rk4,ranocha2018l_2} explicit RK methods, Sun and Shu \cite{sun2019strong} proposed a general framework on strong stability analysis for linear seminegative problems using the energy method. The essential idea of the novel framework \cite{sun2019strong} is to inductively apply a discrete analogue of integration by parts, which was inspired by the stability analysis of PDEs. In particular, it was  proved in \cite{sun2019strong} that all linear RK methods corresponding to $p$th order truncated Taylor expansions are strongly stable if $p\equiv 3~(\bmod~4)$ and are not strongly stable if $p\equiv 1~{\rm or}~2~(\bmod~4)$. 
It is worth noting that the stability analysis in \cite{sun2019strong} is closely related to that of the RK discontinuous Galerkin schemes for linear advection equation by Xu \emph{et al.}~in   \cite{xu20192,xu2020superconvergence}. 
 For nonlinear or nonautonomous problems, the requirement for strong stability may lead to order barriers  \cite{ranocha2021strong,ranocha2020energy}. Remedy approaches to enforce strong stability were also studied recently, including the relaxation RK methods in \cite{ketcheson2019relaxation, ranocha2020relaxation, ranocha2020relaxationH} and the stabilization with artificial viscosity in \cite{offner2020analysis,sun2021enforcing} and references therein.

It is worth particularly mentioning those implicit RK methods associated with 
 the Pad\'e approximations, which are the optimal rational approximation to the exponential function for given degrees of the numerator and denominator. 
 	The proof of A-stability of the diagonal Pad\'e approximations may be dated back to  \cite{birkhoff1965discretization}. Then it was shown that the first and the second subdiagonals in the Pad\'e table are also A-stable \cite{ehle1973stable,ehle1975two}, but all the others are not A-stable \cite{wanner1978order}. It is also worth noting that some of the Pad\'e approximations correspond to the stability functions of certain collocation methods such as the Gauss and Radau methods \cite[Table 5.13]{wanner1996solving}. The analysis of algebraic stability for those collocation methods \cite{hairer1981algebraically} could also lead to the $L^2$-stability of the corresponding Pad\'e approximations.

In this paper, we generalize the stability analysis of {\em explicit} RK methods in \cite{sun2019strong} 
and establish a systematic theoretical framework for analyzing {\em general implicit and explicit} RK methods. 
The efforts and novelty of this paper are summarized as follows. 
\begin{itemize}[leftmargin=*]
	\item  We present a universal framework to derive the energy identity of generic RK method for 
	general linear seminegative systems \cref{eq-odes}. 
	The energy identity provides a precise characterization 
	on how the energy law \eqref{eq:ode-ed} is approximated 
	and whether the energy dissipation property is preserved 
	in the RK discretization. 
	As a result, the established energy identities lead to weak and strong stability criteria of 
	RK methods. 
	\item Our framework is motivated by 
	a series expansion of the {\em continuous} energy law \cref{eq:ode-ed} and 
	a {\em discrete} analogue of {\em integration by parts} technique. 
	Hence we also refer to our energy identities 
	as {\em discrete energy laws}. 
	Our analyses in some special cases establish a close connection between  
	the continuous and discrete energy laws. % (\cref{sec:connection})
	The findings clearly demonstrate the 
	unity of continuous and discrete objects. 
	\item Besides the different motivations, some other aspects of our
	framework are also quite different from those of the eigenvalue analysis and the traditional energy approaches such the algebraic stability analysis. 
	In our discrete energy laws, the energy dissipation is 
	carefully expanded 
	in terms of the proposed semi-norm $\qnm{\cdot}$ associated with the operator $L$. 
	Moreover, our expansion 
	is formulated as a high-order polynomial of the time stepsize $\tau$, which can be compared with 
	the infinite series expansion in the continuous case. 
	%Hence our approach may provide new  perspectives beyond some traditional  approaches. 
	\item Most notably, we discover the unified discrete energy law for all the diagonal Pad\'e approximations of arbitrary orders. 
	Such unified energy law is established based on an analytical Cholesky type decomposition of 
	a class of symmetric matrices. The structure of the matrices is extremely complicated and their elements  involve complex summations of factorial products; see \cref{eq:gammaPade2}. 
	As a result, the discovery of the unified energy law and 
	the proof of the decomposition are highly nontrivial and challenging; see \cref{thm:pade-cholesky} 
	and its proof in \cref{sec:choleskyPade}. 
	Besides, our analyses involve the construction of very technical 
	combinatorial identities and some novel techniques from the theory of hypergeometric series, which seem to be rarely used in previous RK stability analyses and  may shed new lights on  
	%motivate 
	future developments in this direction. 
	\item It is worth noting that the proposed framework applies to a generic RK method, which can be either 
	implicit or explicit, unconditionally stable (A-stable) or conditionally stable (not A-stable). 
	We provide several specific examples of implicit methods in \cref{sec:ex} to further understand the proposed discrete energy laws. A few numerical examples are also given in \cref{sec:num} to confirm the theoretical results. 
\end{itemize}

The paper is organized as follows. We study the continuous energy law in 
\cref{sec:eeode} and present the systematic theoretical framework in \cref{sec:RK} to derive the discrete energy laws of general RK methods and the stability analysis. Examples on implicit RK methods are given in \cref{sec:ex}. We derive the unified discrete energy law of diagonal Pad\'e approximations in 
\cref{sec:pade} and present numerical results in \cref{sec:num} before conclusions in \cref{sec:conclusions}. For better readability, some technical proofs are presented in the appendices.

%Examples on implicit RK methods and the diagonal Pad\'e approximations are give in \cref{sec:ex} and the proof of a key lemma for the diagonal Pad\'e approximations is given in \cref{sec:choleskyPade}. We provide numerical validations on some of our results in \cref{sec:num} and conclusions are given in \cref{sec:conclusions}.

\section{Energy law at continuous level}\label{sec:eeode}

In this section, we derive a series expansion of the continuous energy law \eqref{eq:ode-ed} for the linear seminegative system \eqref{eq-odes}. The main result is given below.

\begin{theorem}\label{thm:ode-ed-series}
	The energy law of the linear seminegative problem \eqref{eq-odes} has the series expansion
	\begin{equation}\label{eq:ode-ed-series}
	\nm{u(t^n+\tau)}^2 - \nm{u(t^n)}^2 = -\sum_{k = 0}^\infty \widehat{\mylambda}_k \tau^{2k+1}\qnm{L^k \widehat{u}^{(k)}}^2,
	\end{equation}
	where 
	\begin{equation}\label{eq:uhatk}
	\widehat{u}^{(k)} = \sum_{j = k }^\infty \widehat{\mu}_{k,j}(\tau L)^{j-k} u(t^n),
	\end{equation}
	with $\widehat{\mylambda}_k$ and $\widehat{\mu}_{k,j}$ defined by  
	\begin{equation}\label{eq:lambdamu_hat}
		\widehat{\mylambda}_k :=  \frac{(k!)^2}{(2k)!(2k+1)!}, 
		\qquad 	\widehat{\mu}_{k,j} := \frac{(2k+1)!j!}{k!(j-k)!(k+j+1)!}
		 \qquad \forall k,j\in \mathbb N,~j\ge k.
	\end{equation}
%	\begin{subequations}\label{eq:lambdamu_hat}
%	\begin{align}
%	\widehat{\mylambda}_k =& \frac{(k!)^2}{(2k)!(2k+1)!} \qquad \forall k\geq 0,\label{eq:lambda_hat}\\
%	\widehat{\mu}_{k,j} =& \frac{(2k+1)!j!}{k!(j-k)!(k+j+1)!} \qquad \forall j\geq k\geq 0.\label{eq:mu_hat}
%	\end{align}
%	\end{subequations}
\end{theorem}

The significance of the expansion \cref{eq:ode-ed-series} lies in that each term in the expansion clearly shows the energy dissipation order with respect to $\tau$. This will help to gain some 
insights on deriving similar expansions for the discrete energy laws of RK methods in \cref{sec:RK}.  \cref{thm:ode-ed-series} will also be useful for establishing a connection between the continuous energy law and the discrete energy laws in \cref{sec:connection}. 
It is also worth noting that the infinite series $\widehat{u}^{(k)}$ in \cref{eq:uhatk} is well-defined, because 
\begin{align*}
	\nm{\widehat{u}^{(k)}} \leq \frac{(2k+1)!}{k!}\sum_{j=k}^\infty \frac{ (\tau\nm{ L })^{j-k} }{(j-k)!}  \nm{u(t^n)} = \frac{(2k+1)!}{k!}e^{\tau \nm{L}}\nm{u(t^n)}<\infty,
\end{align*}
where and hereafter the operator norm of $L$ is defined as $\|L\|:=\sup\{L v:~\| v \|\le 1, v \in V \}$.
 
%a precise characterization of

The proof of \cref{thm:ode-ed-series} is fairly technical and is based on the following two lemmas. 
To improve the readability of the paper, we place 
the detailed proof of \cref{thm:ode-ed-series} in \cref{sec:proof:ode-ed-series}, right after 
the proofs of \cref{lem:quadraticform,lem:cholHilbert} in respectively 
 \cref{sec:proof:quadraticform,sec:proof:cholHilbert}. Note that \cref{lem:quadraticform} 
 will also be useful in deriving the discrete energy laws in \cref{sec:RK}. 

\begin{lemma}\label{lem:quadraticform}
	Let $N$ be a non-negative integer. Assume that the matrix $\pmb{\Upsilon} = (\gamma_{i,j})_{i,j = 0}^\myN$ is negative semidefinite with the Cholesky type decomposition $\pmb{\Upsilon} = -\mathbf{U}^\top \bfmyLambda \mathbf{U}$, where $\mathbf{U} = (\mu_{k,j})_{k,j = 0}^{\myN}$ is an upper triangular matrix and $\mathbf{\bf D} = \diag(\{\mylambda_k\}_{k=0}^\myN)$ is a diagonal matrix with nonnegative entries. Then for any $v \in V$, it holds that 
	\begin{equation}\label{eq:energyexp-odes}
		\sum_{i  = 0}^{\myN} \sum_{j  = 0}^{\myN} \gamma_{i,j}\tau^{i+j+1} \qip{L^i v, L^j v} = -\sum_{k = 0}^\myN \mylambda_k\tau^{2k+1} \qnm{L^k v^{(k)}}^2\leq 0, 
	\end{equation}
	where $v^{(k)} = \sum_{j = k}^\myN \mu_{k,j} (\tau L)^{j-k} v$. 
\end{lemma}

\begin{lemma}\label{lem:cholHilbert}
	Let $\pmb{\widehat{\Upsilon}} = (\widehat{\gamma}_{i,j})_{i,j = 0}^\myN$ and $\widehat{\gamma}_{i,j} = - 
	\frac{1}{ i! j ! (i+j+1) }$. Then it holds that 
	\begin{equation*}
		\pmb{\widehat{\Upsilon}} = - \mathbf{\widehat{U}}^\top \bfhatmyLambda \mathbf{\widehat{U}}, 
	\end{equation*}
	where $\bfhatmyLambda= \diag(\{\widehat{\mylambda}_k\}_{k = 0}^\myN)$ is a diagonal matrix with $\widehat{\mylambda}_{k}$ defined in \eqref{eq:lambdamu_hat}, and $\mathbf{\widehat{U}} = (\widehat{\mu}_{k,j})_{k,j = 0}^\myN$ is an upper triangular matrix with $\widehat{\mu}_{k,j}$ defined in \eqref{eq:lambdamu_hat} for $j\ge k$ and $\widehat{\mu}_{k,j}=0$ for $j<k$.
\end{lemma}

\begin{remark}
The energy decay property $\nm{u(t^n+\tau)}^2-\nm{u(t^n)}^2 = \nm{e^{\tau L} u(t^n)}^2 -  \nm{u(t^n)}^2 \leq 0$ can be equivalently expressed as $(e^{\tau L})^\top e^{\tau L}-I\leq O$ is negative semidefinite. \cref{thm:ode-ed-series} gives a more precise characterization of this property by expanding it into an infinite series of negative semidefinite operators
\begin{equation}\label{eq:coercivity}
(e^{\tau L})^\top e^{\tau L} - I =  \sum_{k = 0}^\infty\widehat{\mylambda}_k \tau^{2k+1}\widehat{U}_k^\top(L^\top+L)\widehat{U}_k \leq O, \quad \mbox{with}~~ \widehat{U}_k := L^k \sum_{j = k}^\infty \widehat{\mu}_{k,j}(\tau L)^{j-k},
\end{equation}
where $\widehat{\mylambda}_k$ and $\widehat{\mu}_{k,j}$ are defined in \eqref{eq:lambdamu_hat} and $L^\top$ is the adjoint operator of $L$. 
%\end{corollary}
The identity \cref{eq:coercivity} directly follows from \cref{eq:ode-ed-series} in \cref{thm:ode-ed-series}, by noting that $u(t^n)$ can be arbitrarily taken in the space $V$. 
%is a straightforward 
%corollary of \cref{thm:ode-ed-series}. 
\end{remark}

%The proof \cref{cor:coercivity} 
%Now we would like to reformulate this term.

\section{Discrete energy laws and stability of Runge--Kutta methods}\label{sec:RK}
We consider the RK discretizations to the seminegative system \eqref{eq-odes}. 
Our goal is to establish a unified framework for deriving the discrete energy laws satisfied by the numerical solutions of the RK methods. 
The discrete energy laws are analogues of the continuous energy law \eqref{eq:ode-ed-series}, and will be very useful for understanding and analyzing the stability of RK methods. 

In general, an RK method 
for the linear autonomous system \cref{eq-odes} can be formulated as 
\begin{equation}\label{eq:RK}
	u^{n+1} = {\mathcal R}(\tau L) u^n,
\end{equation}
where $u^n$ denotes the numerical solution at the $n$th time level $t=t^n$, and $\tau=t^{n+1}-t^n$ is the time stepsize. Here  $\mathcal R(Z)$ is the stability function corresponding to a rational approximation of $e^{Z}$ given by 
\begin{equation}\label{eq:stabilityfunc}
	{\mathcal R}(Z ) = ({\mathcal Q}(Z))^{-1} {\mathcal P}(Z),
\end{equation}
with $\mathcal{P}(Z)$ and $\mathcal{Q}(Z)$ being $s_p$th and $s_q$th order polynomials of $Z$, namely, 
\begin{subequations}\label{eq:ndpoly}
\begin{align}
\mathcal{P}(Z) =& \sum_{i = 0}^{s} \theta_i Z^i, \qquad \text{  with  } \theta_i = 0 \text{ for } i> s_p,\\
\mathcal{Q}(Z) =& \sum_{i = 0}^{s} \vartheta_i Z^i, \qquad \text{  with  } \vartheta_i = 0 \text{ for }  i> s_q, 
\end{align}
\end{subequations}
where $s := \max\{s_p,s_q\}$, and a normalization is typically used such that $\theta_0=\vartheta_0=1$. 
For convenience, we denote $P := \mathcal{P}(\tau L)$ and $Q := \mathcal{Q}(\tau L)$. Note that the operators $L$, $P$, and $Q^{-1}$ commute with each other.

\begin{remark}
In the special case that $s_q=0$, namely, $\mathcal{R}(Z)$ is a polynomial approximation of $e^{Z}$, then 
$Q = I$ is the identity operator, and the scheme \cref{eq:RK} is an explicit RK method, whose stability was studied in \cite{sun2019strong} via the energy approach. When $s_q\ge 1$, the RK method \cref{eq:RK} is implicit, which is the particular focus of the present paper. 
\end{remark}

\subsection{Discrete energy laws} We first give a lemma on the energy change of the RK method \cref{eq:RK}. %By some simple algebraic manipulations
\begin{lemma} \label{lem:IIF}
	The solution of the RK method \cref{eq:RK} satisfies the following identity
	\begin{equation}\label{eq:energyeq-1}
	\nm{u^{n+1}}^2 - \nm{u^n}^2 = \sum_{i=0}^s \sum_{j=0}^s \alpha_{i,j} \tau^{i+j} \ip{L^i w^n, L^j w^n}, 
	\end{equation} 
where $w^n := Q^{-1} u^n$ and $\alpha_{i,j} := \theta_{i}\theta_j - \vartheta_i\vartheta_j$. 
\end{lemma}
\begin{proof}
	Some simple algebraic manipulations give
	\begin{equation}\label{eq:enenrgyeq-0}
	\begin{aligned}
		\nm{u^{n+1}}^2 =& \nm{Q^{-1}Pu^n}^2 = \nm{PQ^{-1}u^n}^2
		=  \nm{u^n}^2 + \nm{PQ^{-1}u^n}^2 - \nm{QQ^{-1}u^n}^2\\
		=&  \nm{u^n}^2 + \nm{Pw^n} - \nm{Qw^n}^2.
	\end{aligned}
	\end{equation}
	Note that 
	\begin{equation*}
	\nm{P w^n}^2 = \ip{\sum_{i = 0}^s \theta_i (\tau L)^i w^n, \sum_{j = 0}^s \theta_j (\tau L)^j w^n} = \sum_{i = 0}^{s} \sum_{j = 0}^{s} \theta_i\theta_j  \tau^{i+j}\ip{L^i w^n, L^j w^n}, 
	\end{equation*}
	and similarly $\nm{Q w^n}^2 = \sum_{i = 0}^{s} \sum_{j = 0}^{s} \vartheta_i\vartheta_j  \tau^{i+j}\ip{L^i w^n, L^j w^n}$. 
	Substituting these expansions into \eqref{eq:enenrgyeq-0} gives \cref{eq:energyeq-1} and 
	completes the proof. 
\end{proof}

However, from the energy identity \cref{eq:energyeq-1}, 
it is very difficult to judge whether the energy $\nm{u^n}^2$ always decays or not, because the sign of 
each term $\ip{L^i w^n, L^j w^n}$ in \cref{eq:energyeq-1} is unclear and indeterminate. 
In order to address this difficulty, we would like to 
reformulate $\ip{L^i w^n,L^j w^n}$ into a linear combination of some terms of form $\nm{L^k w^n }^2$ and 
$\qip{L^k w^n,L^{l} w^n}$. 
Such a reformulation procedure can be completed by repeatedly using a 
discrete analogue of the integration by parts formula
\begin{equation}\label{eq-ibp}
	\ip{w,Lv}  
	= - \ip{Lw,v} - \qip{w,v},
\end{equation}
which follows from the definition \cref{eq:Def-semi-inner} and gives 
\begin{equation}\label{eq-prop-1}
	\ip{L^i w^n,L^j w^n} = 
	\begin{cases}
		\nm{L^i w^n }^2,& j = i,\\
		-\frac{1}{2} \qnm{L^i w^n }^2,&j = i+1,\\
		-\ip{L^{i+1} w^n,L^{j-1} w^n } - \qip{L^i w^n,L^{j-1} w^n},& \text{otherwise.}
	\end{cases}
\end{equation}
See \cite[Proposition 2.1]{sun2019strong} for a proof of \cref{eq-prop-1}. 
It is worth noting that such a discrete version of integration by parts is inspired by 
approximating the spatial derivative $\partial_x$ with $L$. 
%, which typically arises in the semi-discrete schemes for the linear conservation law $\partial_t u = \partial_x u$. 

%
%In the context of solving the linear conservation law $u_t = u_x$, $L$ is an approximation of $\partial_x$. It is natural to perform integration by parts to reassemble $\ip{L^iv,L^jv}$ terms. At the discrete level, this integration by parts formula is given by
%\begin{equation}\label{eq-ibp0}
%	\ip{Lv,w} 
%	= -\ip{v,Lw} - \qip{v,w},
%\end{equation}
%which gives \cite[Proposition 2.1]{sun2019strong}
%\begin{equation}\label{eq-prop-10}
%	\ip{L^i u^n,L^j u^n} = \left\{
%	\begin{array}{ll}
%		\nm{L^i u^n }^2,& j = i\\
%		-\frac{1}{2} \qnm{L^i u^n }^2,&j = i+1\\
%		-\ip{L^{i+1} u^n,L^{j-1} u^n } - \qip{L^iu^n,L^{j-1}u^n},& \text{otherwise}
%	\end{array}
%	\right..
%\end{equation}

Recursively applying \eqref{eq-prop-1} to reformulate the terms $\ip{L^iv,L^jv}$ in \cref{eq:energyeq-1}, 
we obtain an energy identity in the following form.

%inductively will lead to an energy equality of the following form. 

\begin{lemma}\label{lem:abc}
	For the solution of the RK method \cref{eq:RK}, the following identity holds:
	\begin{equation}\label{eq:energyexp}
		\nm{u^{n+1}}^2 - \nm{u^n}^2 = \sum_{k = 0}^s \beta_k\tau^{2k}\nm{L^k w^n}^2 + \sum_{i = 0}^{s-1}  \sum_{j = 0}^{s-1}  \gamma_{i,j}\tau^{i+j+1}\qip{L^i w^n, L^j w^n},
	\end{equation}
	where $\beta_k$ and $\gamma_{i,j}$ are computed from the values of $\alpha_{i,j} = \theta_i\theta_j - \vartheta_i\vartheta_j$ via the formulae
	\begin{align}
		\beta_k =& \sum_{\ell = \max\{0,2k-s\}}^{\min\{2k,s\}} \alpha_{\ell,2k-\ell}(-1)^{k-\ell}, \label{eq:beta}\\
		\gamma_{i,j} =& \sum_{\ell = \max\{0,i+j+1-s\}}^{\min\{i,j\}} (-1)^{\min\{i,j\}+1-\ell}\alpha_{\ell,i+j+1-\ell}.\label{eq:gamma}
	\end{align}
\end{lemma}
The coefficients $\beta_k$ and $\gamma_{i,j}$ in \cref{lem:abc} are obtained by the computer algorithm in  \cite[Algorithm 2.1]{sun2019strong}. While these formulae can also be shown by mathematical induction, an alternative proof using combinatorial identities will be given in \cite{GopaSun2022}, and the details are omitted here. 
We remark that for a given RK method,  $\{\theta_i\}$ and $\{\vartheta_i\}$ are given, 
and $\{\beta_k\}$ and $\{\gamma_{i,j}\}$ are determined by \cref{eq:beta}--\cref{eq:gamma}.

Note that the first term at the right-hand side of \cref{eq:energyexp} 
has a similar format as that in the continuous energy law \cref{eq:ode-ed-series}. 
Next, we would like to reformulate the last term of \cref{eq:energyexp} 
by using \cref{lem:quadraticform}.   Define 
%For convenience, we define 
\begin{equation*}
\mathbf{B} := \diag(\{\beta_k\}_{k=0}^{s})\quad \text{and} \quad \pmb{\Upsilon} := (\gamma_{i,j})_{i,j=0}^{s-1}
\end{equation*}
with $\beta_k$ and $\gamma_{i,j}$ given by \cref{eq:beta}--\cref{eq:gamma}, respectively. 
%where $\pmb{\Upsilon}$ is a symmetric matrix defined by \eqref{eq:gamma}. %We also use the notation $\mathbf{A}\leq 0$ to state that $\mathbf{A}$ is negative semidefinite. 
However, for some RK methods the symmetric matrix $\pmb{\Upsilon}$ is not necessarily negative semidefinite, so that 
its Cholesky type decomposition required in \cref{lem:quadraticform}  
may not exist. 
In case this happens, one can overcome such a problem by subtracting a diagonal matrix. 
We finally obtain the following practical discrete energy law \cref{eq:enegyexp-cholesky} for general RK methods.

%Note that the second term in \eqref{eq:energyexp} is in a similar format as that of \eqref{eq:energyexp-odes}. Hence we would like to apply \cref{lem:quadraticform} to reformulate the energy equality \eqref{eq:energyexp}. To make sure that the Cholesky type decomposition exists, we need to add and subtract a matrix so that the corresponding coefficient matrix is negative semidenite. In the end, we have the following energy equality for RK methods. 
\begin{theorem}[Energy identity]\label{thm:energy}
	Assume that $\pmb{\tilde{\Upsilon}} = \pmb{\Upsilon} -\pmb{\Delta}$ is negative semidefinite for some diagonal matrix
	$
	\pmb{\Delta} = \diag(\left\{\delta_k\right\}_{k=0}^{s-1})
	$ with $\delta_k \ge 0$ for $0 \le k \le s-1$, so that  
	the symmetric matrix $\pmb{\tilde{\Upsilon}}$ admits the Cholesky type decomposition $\pmb{\tilde{\Upsilon}} = - \mathbf{\tilde{U}}^\top \bftildemyLambda\mathbf{\tilde U}$,  where $\mathbf{\tilde U} = (\tilde{\mu}_{k,i})_{k,i = 0}^{s-1}$ is an upper triangular matrix with $\mu_{k,k} = 1$ and $\bftildemyLambda = \diag(\{\tilde{\mylambda}_{k}\}_{k=0}^{s-1})$ with $\tilde{\mylambda}_{k}\geq 0$ for $0\leq k\leq s-1$. 
	The solution of the RK method \eqref{eq:RK} satisfies the following energy identity:
	\begin{equation}\label{eq:enegyexp-cholesky}
	\nm{u^{n+1}}^2 - \nm{u^n}^2 = \sum_{k = 0}^s \beta_k \tau^{2k} \nm{L^kw^n}^2 - \sum_{k = 0}^{s-1} \tilde{\mylambda}_k \tau^{2k+1} \qnm{L^ku^{(k)}}^2 + \sum_{k = 0}^{s-1} \delta_k \tau^{2k+1} \qnm{L^k w^n}^2, 
	\end{equation}
	where $u^{(k)} := \sum_{j=k}^s\tilde{\mu}_{k,j}(\tau L)^{j-k}w^n= \sum_{j=k}^s\tilde{\mu}_{k,j}(\tau L)^{j-k} Q^{-1} u^n$.
\end{theorem}
\begin{proof}
	Denote $\pmb{\tilde{\Upsilon}} =: (\tilde{\gamma}_{i,j})_{i,j=0}^{s-1}$. Then it follows from \cref{eq:energyexp} and $\pmb{\Upsilon} = \pmb{\tilde{\Upsilon}} + \pmb{\Delta}$ that 
	\begin{align*}
	\nm{u^{n+1}}^2 - \nm{u^n}^2 &= \sum_{k = 0}^s \beta_k \tau^{2k} \nm{L^kw^n}^2 + \sum_{i = 0}^{s-1} \sum_{j = 0}^{s-1} \gamma_{i,j} \tau^{i+j+1} \qip{ L^i w^n, L^j w^n }\\
	&= \sum_{k = 0}^s \beta_k \tau^{2k} \nm{L^kw^n}^2 + \sum_{i = 0}^{s-1} \sum_{j = 0}^{s-1} \tilde \gamma_{i,j} \tau^{i+j+1} \qip{ L^i w^n, L^j w^n } + \sum_{k = 0}^{s-1}\delta_k\tau^{2k+1}\qnm{L^k w^n}^2.  
	\end{align*}
	Using \cref{lem:quadraticform} to reformulate the second term yields \eqref{eq:enegyexp-cholesky}.
\end{proof}

Examples of the discrete energy law \cref{eq:enegyexp-cholesky} for several 
specific RK schemes will be given in \cref{sec:ex}. 

\subsection{Stability analysis} 
This subsection applies the discrete energy law \cref{eq:enegyexp-cholesky} 
in \cref{thm:energy} to analyze the stability of RK methods.

First, consider a special case: both $\pmb{\Upsilon}$ and $\mathbf{B}$ 
are negative semidefinite. We obtain the unconditional strong stability of the corresponding RK method from the discrete energy law \cref{eq:enegyexp-cholesky}.

%In this case, we can take $\pmb{\Delta}={\bf O}$, and the discrete energy law \cref{eq:enegyexp-cholesky}  %immediately
%implies the unconditional strong stability of the corresponding RK method. 

%If $\pmb{\Upsilon}$ and $\mathbf{B}$ are both negative semidefinite, we can take $\delta_k = 0$ and $\beta_k,\tilde{d}_k \leq 0$ for all $k$. As a result, \cref{thm:energy} implies that the corresponding RK methods are unconditionally strongly stable. 

\begin{theorem}[Unconditional strong stability]\label{thm:unstab}
	If the RK method \eqref{eq:RK} satisfies that 
	 $\pmb{\Upsilon}$ and $\mathbf{B}$ are both negative semidefinite, then the RK method \eqref{eq:RK} is unconditionally strongly stable, namely,  
	\begin{equation}\label{eq:USS}
	\nm{u^{n+1}}^2\leq \nm{u^n}^2 \qquad \forall \tau\geq 0. 
	\end{equation}
\end{theorem}

\begin{proof}
When $\pmb{\Upsilon}$ is negative semidefinite, \cref{thm:energy} holds with $\pmb{\Delta}={\bf O}$, namely, we can take $\delta_k =0$, so that the energy identity \cref{eq:enegyexp-cholesky} becomes 
	\begin{equation*} 
	\nm{u^{n+1}}^2 - \nm{u^n}^2 = \sum_{k = 0}^s \beta_k \tau^{2k} \nm{L^kw^n}^2 - \sum_{k = 0}^{s-1} \tilde{\mylambda}_k \tau^{2k+1} \qnm{L^ku^{(k)}}^2.
\end{equation*}
This yields \cref{eq:USS}, because $\tilde{\mylambda}_k \ge 0$ and $\beta_k\le 0$  as $\mathbf{B}$ is negative semidefinite. %The proof is completed. 
\end{proof}

In the rest of this section, we discuss the general case that $\pmb{\Upsilon}$ is not necessarily negative semidefinite, and we shall use the energy law \cref{eq:enegyexp-cholesky} to derive several 
stability criteria  under some constraint on the time stepsize $\tau$. 
For simplicity, 
we will denote $\tau \nm{L} =: \mycfl$, and use the notations $\mycfl_0$ and $C$ to represent generic positive constants, which are independent of $\tau$ and $\nm{L}$ but may depend on $\theta_i$, $\vartheta_i$, and $s$. The values of $\mycfl_0$ and $C$ may vary at different places.

\begin{lemma}[Energy estimate]\label{lem:eeconditional}
	Let $\zeta$ be the index of the first nonzero element in $\{\beta_k\}_{k=0}^s$. Let $\rho$ be the largest index such that the $\rho$th order principle submatrix $(\gamma_{i,j})_{i,j=0}^{\rho-1}$ is negative semidefinite. There exists a positive constant $c_\rho$ such that 
	\begin{equation}\label{eq:conditionalee}
	\nm{u^{n+1}}^2 - \nm{u^n}^2 \leq \left(\beta_\zeta + \mycfl^2 g_\beta(\mycfl)\right)\tau^{2\zeta}\nm{L^\zeta w^n}^2 +  \mycfl c_\rho(1 + \mycfl^2 g_\rho (\mycfl))\tau^{2\rho}\nm{L^\rho w^n}^2,
	\end{equation}
	where $g_\beta(\lambda) := \sum_{i = 0}^{s-\zeta-1}\beta_{i+\zeta+1}\mycfl^{2i}$ and $g_\rho(\lambda) := \sum_{i = 0}^{s-\rho-2}\mycfl^{2i}$ are polynomials of $\lambda$. 
\end{lemma}

\begin{proof}
	Since the $\rho$th order principle submatrix of $\pmb{\Upsilon}$ is negative semidefinte, there exists a positive constant $c_\rho$ such that the symmetric matrix 
	$$\pmb{\tilde{\Upsilon}}:=\pmb{\Upsilon} - \frac12 \diag\{ \underbrace{0,\dots,0}_{\rho}, \underbrace{c_\rho,\dots c_\rho}_{s-\rho}\} =: \pmb{\Upsilon} - \pmb{\Delta}$$ 
	is negative semidefinite. 
	According to the energy law \cref{eq:enegyexp-cholesky} in 
	\cref{thm:energy}, we have 
	%Here in the subtracted diagonal matrix, the first $\rho$ entries along the diagonal is $0$ and the rest of the diagonal entries are $\frac{c_\rho}2$. Therefore, we can apply \cref{eq:enegyexp-cholesky} to obtain 
	\begin{equation}\label{eq:enegyexp-cholesky3}
		\nm{u^{n+1}}^2 - \nm{u^n}^2 = \sum_{k = 0}^s \beta_k \tau^{2k} \nm{L^kw^n}^2 - \sum_{k = 0}^{s-1} \tilde{\mylambda}_k \tau^{2k+1} \qnm{L^ku^{(k)}}^2 + \frac{c_\rho}{2}\sum_{k = \rho}^{s-1} \tau^{2k+1} \qnm{L^k w^n}^2.
	\end{equation}
	For the first term at the right-hand side of \eqref{eq:enegyexp-cholesky3}, using the Cauchy--Schwarz inequality gives 
	\begin{equation}\label{eq:ee1}
		\begin{aligned}
			\sum_{k = 0}^s \beta_k \tau^{2k} \nm{L^kw^n}^2 =& \sum_{k = \zeta}^s \beta_k  \nm{(\tau L)^kw^n}^2 
			\leq  \sum_{k = \zeta}^s \beta_k \left(\tau \nm{L}\right)^{2(k-\zeta)}\tau^{2\zeta}\nm{L^\zeta w^n}^2\\
			= &\left(\beta_\zeta + \mycfl^2 \sum_{i = 0}^{s-\zeta-1}\beta_{i+\zeta+1}\mycfl^{2i}\right)\tau^{2\zeta}\nm{L^\zeta w^n}^2.
			%\\
			%: =& \left(\beta_\zeta + \mycfl^2 g_\beta(\mycfl)\right)\nm{L^\zeta w^n}^2. 
		\end{aligned}
	\end{equation}
	For the second term, since $\tilde{\mylambda}_k \ge 0$ for $0\leq k\leq s-1$, we have 
	\begin{equation}\label{eq:ee2}
		-\sum_{k = 0}^{s-1} \tilde{\mylambda}_k \tau^{2k+1} \qnm{L^ku^{(k)}}^2 \leq 0.
	\end{equation}
	For the last term, one can again utilize the Cauchy--Schwarz inequality to obtain 
	\begin{equation}\label{eq:ee3}
		\begin{aligned}
			\frac{c_\rho}{2} \sum_{k = \rho}^{s-1} \tau^{2k+1} \qnm{L^k w^n}^2
			\leq c_\rho \sum_{k = \rho}^{s-1} \left(\tau \nm{L}\right)^{2(k-\rho)+1} \nm{(\tau L)^\rho w^n}^2 
			= \mycfl c_\rho  \left(1 + \mycfl^2\sum_{i = 0}^{s-\rho-2}\mycfl^{2i}\right)\tau^{2\rho}\nm{L^\rho u^n}^2.
		\end{aligned}
	\end{equation}
	Combining the estimates in \cref{eq:ee1}--\cref{eq:ee3} with \cref{eq:enegyexp-cholesky3} gives \cref{eq:conditionalee} and completes the proof. 
\end{proof}

%The proof of this lemma can be found in \cref{app:eeconditional}. 

%\begin{remark}
%	Note that the time step constraint $\tau \nm{L}$
%\end{remark}
\begin{theorem}[Conditional stability criteria]\label{thm:conditionalstability}
	Let $\zeta$ and $\rho$ be the indexes defined in \cref{lem:eeconditional} and $\kappa := \min\{2\zeta,2\rho+1\}$. We have the following stability criteria for a generic RK method:
	\begin{enumerate}
		\item The RK method \eqref{eq:RK} is weakly$(\kappa)$ stable, namely,
	$
		\nm{u^{n+1}}^2 \leq (1 + C\mycfl^\kappa) \nm{u^{n}}^2,
		$ 
		under a time step constraint $\mycfl \leq \mycfl_0$ for some positive constant $\mycfl_0$. Furthermore, if $\mycfl^\kappa/\tau$ is bounded, or equivalently, $\tau \nm{L}^{1+1/(\kappa-1)}\leq \mycfl_0$ for some positive constant $\mycfl_0$, then 
		%\begin{equation}\label{eq:eeconditional-weak2}
		$\nm{u^n}^2\leq e^{Ct^n}\nm{u^0}^2$, 
		%\end{equation}
		where $t^n = n \tau$.
		\item If $\zeta\leq \rho$ and $\beta_\zeta<0$, then the RK method \eqref{eq:RK} is strongly stable, namely, 
		%\begin{equation}\label{eq:SS}
		$\nm{u^{n+1}}^2\leq \nm{u^n}^2$, 
		%\end{equation}
		 under a time step constraint $\mycfl \leq \mycfl_0$ for some positive constant $\mycfl_0$. 
		 \item If $\beta_\zeta >0$, then the RK method \eqref{eq:RK} is not strongly stable for a generic seminegative system \cref{eq-odes}, namely, there exist  a linear seminegative operator $L$ and a positive constant $\lambda_0$ such that  
		 $\| {\mathcal R}(\tau L) \|>1$ for any $\lambda \in (0, \lambda_0]$. 
	\end{enumerate}
\end{theorem}

\begin{proof}
	For the first part on the weak$(\kappa)$ stability, we observe that 
	\begin{equation*}
		\nm{u^n} = \nm{Q w^n}  = \nm{w^n + \sum_{k = 1}^s \vartheta_k (\tau L)^k w^n}\geq \left(1- \sum_{k=1}^s |\vartheta_k| (\tau \nm{L})^k \right) \nm{w^n}.
	\end{equation*}
	When $\tau \nm{L} = \mycfl$ is sufficiently small, 
	we have $\nm{u^n}\geq \frac12 \nm{w^n}$ and  $\nm{w^n}\leq 2\nm{u^n}$. 
	It follows that $\tau^{2 k} \nm{L^{k} w^n}^2\leq \mycfl^{2k}\nm{w^n}^2 \leq 4\mycfl^{2k}\nm{u^n}^2$. 
	Similar arguments yield $\beta_\zeta + \mycfl^2 g_\beta(\mycfl) \le 2 |\beta_\zeta|$ and  
	$c_\rho(1 + \mycfl^2 g_\rho (\mycfl)) \le 2 c_\rho$ when $\mycfl$ is sufficiently small. 
	These together with the energy estimate in \cref{lem:eeconditional} imply 
	\begin{align*}
		\nm{u^{n+1}}^2 & \le  \nm{u^n}^2 +  2 |\beta_\zeta| \tau^{2\zeta}\nm{L^\zeta w^n}^2 + 2 c_\rho \mycfl  \tau^{2\rho}\nm{L^\rho w^n}^2
		\\
		& \le  \left(1+ 8 |\beta_\zeta| \mycfl^{2\zeta} + 8c_\rho \mycfl^{2\rho+1}\right)\nm{u^n}^2 \leq (1+C\mycfl^\kappa )\nm{u^n}^2,  
	\end{align*}
    under the constraint $\mycfl \leq \mycfl_0$ for some positive constant $\mycfl_0$. 
    Furthermore, if $\mycfl^\kappa/\tau$ is bounded, we have 
	\begin{equation*}
		\nm{u^{n}}^2 \leq  (1+C\mycfl^\kappa )^n \nm{u^0}^2=   (1+C\mycfl^\kappa )^{
			\mycfl^{-\kappa} \cdot t_n \cdot  \frac{ \mycfl^\kappa }{\tau} } \nm{u^0}^2 \leq e^{Ct^n\mycfl^\kappa/\tau}\nm{u^0}^2\leq e^{Ct^n}\nm{u^0}^2.
	\end{equation*}

We then turn to prove the second part of the theorem. Observe that
$\mycfl g_\beta(\mycfl) \le 1$ and $\mycfl^2 g_\rho (\mycfl) \le 1$ when $\lambda \le \widehat \lambda_0$ for some constant $\widehat \lambda_0$. Thanks to \cref{lem:eeconditional}, when $\zeta\leq \rho$ and $\lambda \le \widehat \lambda_0$ we then have 
\begin{align*}
		\nm{u^{n+1}}^2 - \nm{u^n}^2 & \leq \left(\beta_\zeta + \mycfl \right)\tau^{2\zeta}\nm{L^\zeta w^n}^2 + 2 c_\rho  \mycfl   \tau^{2\rho}\nm{L^\rho w^n}^2
		\\
		& \le \left(\beta_\zeta + \mycfl \right)\tau^{2\zeta}\nm{L^\zeta w^n}^2 + 2 c_\rho \mycfl   \tau^{2\rho}  \nm{L}^{ 2(\rho-\zeta) } \nm{L^\zeta w^n}^2 
		\\
		& \le \left( \beta_\zeta + \mycfl + 2  \mycfl c_\rho \widehat \lambda_0 ^{ 2(\rho-\zeta) }  \right) \tau^{2\zeta}\nm{L^\zeta w^n}^2,
\end{align*}
where the last term is non-positive if $\mycfl \le |\beta_\zeta|/( 1+2 c_\rho \widehat \lambda_0^{ 2(\rho-\zeta)} ) $. We therefore obtain $\nm{u^{n+1}}^2\leq \nm{u^n}^2$ under the constraint $\mycfl \le \mycfl_0$ with 
$ \mycfl_0:=\min \{ \widehat \lambda_0, |\beta_\zeta|/( 1+2 c_\rho \widehat \lambda_0^{ 2(\rho-\zeta)} ) \} $.

For the third part, one can consider a special operator $L$ satisfying $L^\zeta Q^{-1}\neq O$ but $L^\top + L= O$, so that the last term in \cref{eq:energyexp} vanishes. 
It then follows from \cref{lem:abc} that  
\begin{align*}%\label{eq:energyexp111}
	\nm{{\mathcal R}(\tau L)  u^{n}}^2 - \nm{u^n}^2 %&= \sum_{k = 0}^s \beta_k\tau^{2k}\nm{L^k w^n}^2 
	= \sum_{k = \zeta}^s \beta_k\tau^{2k}\nm{L^k Q^{-1} u^n}^2 
\ge  \left( \beta_\zeta - \sum_{k = \zeta+1}^s |\beta_k| \lambda  ^{2k} \right) \tau^{2\zeta} \nm{L^\zeta Q^{-1} u^n}^2.
\end{align*}
Hence when $\lambda$ is sufficiently small, we have $\nm{{\mathcal R}(\tau L)  u^{n}} / \nm{u^n}>1$ for all $u^n$ satisfying $L^\zeta Q^{-1}u^n \neq 0$, which implies $\nm{{\mathcal R}(\tau L)}>1$. The proof is completed.
%Since $u^n$ can be arbitrary in $V$, we have $\nm{{\mathcal R}(\tau L)} \ge 1+ \left( \beta_\zeta - \sum_{k = \zeta+1}^s |\beta_k| \lambda ^{2k} \right) \tau^{2\zeta} \nm{L^\zeta Q^{-1} }^2 >1$ provided that $\lambda$ is sufficiently small. The proof is completed. 
\end{proof}

\begin{remark}
	If system \cref{eq-odes} is obtained from spatially semi-discrete schemes for linear
	hyperbolic conservation laws, then we have $\nm{L} = \mathcal{O}(h^{-1})$, where $h$ is the spatial mesh size. In this case, the time step constraint $\lambda = \tau \nm{L} \leq \lambda_0$ in \cref{thm:conditionalstability} becomes the Courant–Friedrichs–Lewy (CFL) condition $\tau \leq Ch$.  The time step constraint for weak$(\kappa)$ stability, $\tau \nm{L}^{1+1/(\kappa-1)}\leq \mycfl_0$,  becomes $\tau \leq C h^{1+1/(\kappa-1)}$.  
\end{remark}

\begin{remark}
	The stability analyses in \cref{thm:conditionalstability} and \cite{sun2019strong} are closely connected with the $L^2$-stability analysis of RK discontinuous Galerkin schemes for the linear advection equation by Xu \emph{et al.}~in \cite{xu20192,xu2020superconvergence},  
	%Different notations and terminologies are used. In \cite{sun2019strong}, we introduced the so-called leading index and denote it by $k^*$. This index coincides with the termination index $\zeta$ in \cite{xu20192}. $\rho$ is introduced in \cite{xu20192} as the contribution index. In this paper, we use the notations in \cite{xu20192,xu2020superconvergence}. Also in \cite{xu20192,xu2020superconvergence}, 
	where the weak$(\kappa)$ stability was systematically studied, and  
	 the property $\nm{u^{n+1}}^2\leq \nm{u^n}^2$ was called monotonicity stability in \cite{xu20192,xu2020superconvergence}. 
	%the property $\nm{u^{n+1}}^2\leq \nm{u^n}^2$ is referred to as the monotonicity stability. 
	%The strong stability in \cite{xu20192,xu2020superconvergence} refers to $\nm{u^n}^2\leq \nm{u^0}^2$. The weak$(\kappa)$ stability is introduced following the definition in \cite{xu20192,xu2020superconvergence}.
		The discussions in \cite{xu20192,sun2019strong,xu2020superconvergence} 
		%was focused on the explicit RK methods which reduce to truncated Taylor expansions for linear problems}, and 
		%the discussions in \cite{sun2019strong} were 
		were focused on explicit RK methods. 
		In the present paper, our framework,  including the discrete energy laws and stability results,  
		applies to both general implicit and explicit RK methods. 
	%The conclusions of our stability results 
	 %\cref{thm:conditionalstability} generalize the existing results, such as \cite[Theorem 2.7]{sun2019strong} (focused on a generic explicit RK method) and \cite[Theorem 3.1]{xu20192} (focused on 
	% the RK methods reducing to truncated Taylor expansions), to general explicit and implicit 
	 %RK methods. 
	 % explicit RK methods to both the explicit and implicit RK methods. %Here we adopt notations in \cite{xu20192} for better clarity. 
	%We consider a generic RK method, while the results \cite{xu20192,xu2020superconvergence} are focused on Taylor truncated RK method.  
	%those RK methods reducing to truncated Taylor expansions, namely $\theta_i=1/i!$ $Q=I$. 
\end{remark}

\section{Examples on discrete energy laws}\label{sec:ex}
This section gives several specific examples of implicit methods to further illustrate the proposed discrete energy law \cref{eq:enegyexp-cholesky} in \cref{thm:energy}.

\subsection{Examples of unconditional strong stability} We first use our framework to derive the energy identity for several A-stable implicit RK schemes. For these schemes, the conditions in \cref{thm:unstab} are satisfied so that the strong stability holds without any time step constraint. 
\begin{example}[Euler backward method] 
	The stability function of this method is ${\mathcal R}(Z ) = (I-Z)^{-1}$. 
	%The Butcher tableau and corresponding stability function of this method are  respectively given by 
%				$$\begin{array}{c|c}
%		1 & 1 \\
%		\hline
%		& 1 \\
%	\end{array},\qquad \qquad {\mathcal R}(Z ) = (I-Z)^{-1}.$$ 
Using \cref{lem:abc} gives  
$
\mathbf{B}={\diag}\{ 0, -1 \}$ and $\pmb{\Upsilon}=(-1) = -\mathbf{U}^\top \bfmyLambda \mathbf{U}$ with $\bfmyLambda = (1)$ and $\mathbf{U} = (1)$. 
Since $w^n =Q^{-1} u^n = \mathcal{R}(\tau L) u^n = u^{n+1}$, according to \cref{thm:energy} we obtain the energy law as  
\begin{equation*}
	\nm{u^{n+1}}^2 - \nm{u^n}^2 = -\tau^2\nm{Lu^{n+1}}^2 -\tau \qnm{u^{n+1}}^2.
\end{equation*}
\end{example}

\begin{example}[Crank--Nicolson method and implicit midpoint method] 
%	The Butcher tableaux of these two methods are
%	$$\begin{array}{c|c}
%		\frac{1}{2} & \frac{1}{2} \\
%		\hline
%		& 1 \\
%	\end{array}\quad \quand \quad 
%\begin{array}{c|cc}
%	0 & 0 & 0 \\
%	1 & \frac{1}{2} & \frac{1}{2} \\
%	\hline
%	& \frac{1}{2} & \frac{1}{2} \\
%\end{array},
%$$
%	respectively, and their stability functions are both 
    The stability functions of these two methods are both  
	$
	{\mathcal R}(Z ) = \left(I-\frac{Z}2\right)^{-1} \left(I+\frac{Z}2\right).
	$ 
	By \cref{lem:abc}, we have 
	$
	\mathbf{B}={\diag}\{ 0, 0 \}$ and $\pmb{\Upsilon}=(-1) = -\mathbf{U}^\top \bfmyLambda \mathbf{U}$ with $\bfmyLambda = (1)$ and $\mathbf{U} = (1)$. 
%	 Since
%	\begin{equation*}
%	\pmb{\Upsilon} = -\mathbf{U}^\top \bfmyLambda \mathbf{U}, \quad \bfmyLambda = (1), \quand \mathbf{U} = (1),
%	\end{equation*}
	%one can get
	%It follows from Theorem \ref{thm:energy} that 
	According to \cref{thm:energy}, we obtain the energy identity 
	\begin{equation*}
		\nm{u^{n+1}}^2 - \nm{u^n}^2 =-\tau \qnm{w^n}^2.
	\end{equation*}
\end{example}

\begin{example}[Qin and Zhang \cite{qin1992symplectic}] 
	The Butcher tableau and stability function of this method are  
	$$\begin{array}{c|cc}
		\frac{1}{4} & \frac{1}{4} & 0 \\
		\frac{3}{4} & \frac{1}{2} & \frac{1}{4} \\\hline
		& \frac{1}{2} & \frac{1}{2} \\
	\end{array}
	%and the corresponding stability function is 
	%$$
	\qquad \qquad
	{\mathcal R}(Z ) = \left(I-\frac{Z}{2}+\frac{Z^2}{16} \right)^{-1} \left(I+\frac{Z}{2}+\frac{Z^2}{16} \right).
	$$
	According to \cref{lem:abc}, we have $\mathbf{B}={\diag}\left\{ 0, 0, 0 \right\}$ 
	and %$\pmb{\Upsilon}={\diag}\left\{ -1,-1/16 \right\}$, which satisfies 
	%$$
	%\mathbf{B}={\diag}\left\{ 0, 0, 0 \right\} \quand \pmb{\Upsilon}=\begin{pmatrix}
	%	-1 & 0\\
	%	0 & -\frac1{16}
	%\end{pmatrix}.$$
	%Furthermore, we have
	\begin{equation*}
	\pmb{\Upsilon} = {\diag}\left\{ -1,-1/16 \right\} = -\mathbf{U}^\top \bfmyLambda \mathbf{U}, \quad \bfmyLambda =\diag \left\{1,\frac{1}{4}\right\}, \quand  
	\mathbf{U} = \left(
	\begin{array}{cc}
	1 & 0 \\
	0 & 1 \\
	\end{array}
	\right). 
	\end{equation*}  
	Thanks to \cref{thm:energy}, we obtain the corresponding energy identity
	\begin{equation*}
		\nm{u^{n+1}}^2 - \nm{u^n}^2 = -\tau\qnm{w^n}^2 -\frac{1}{16}\tau^3\qnm{Lw^n}^2.
	\end{equation*}
\end{example}

\begin{example}[Kraaijevanger and Spijker \cite{KRAAIJEVANGER198971}] 
	The Butcher tableau and corresponding stability function of this method are 
	$$\begin{array}{c|cc}
		\frac{1}{2} & \frac{1}{2} & 0 \\
		\frac{3}{2} & -\frac{1}{2} & 2 \\
		\hline
		& -\frac{1}{2} & \frac{3}{2} \\
	\end{array} \qquad \qquad  {\mathcal R}(Z ) = \left( I-\frac{5 Z}{2}+Z^2 \right)^{-1} \left( I-\frac{3 Z}{2}+\frac{Z^2}{2} \right).$$ 
%	and the corresponding stability function is 
%	$$
%	{\mathcal R}(Z ) = \left( I-\frac{5 Z}{2}+Z^2 \right)^{-1} \left( I-\frac{3 Z}{2}+\frac{Z^2}{2} \right).
%	$$
	According to \cref{lem:abc}, we have $\mathbf{B}={\diag}\left\{ 0, -3, -\frac{3}{4} \right\}$ and 
	$$
	\pmb{\Upsilon}=\begin{pmatrix}
				-1 & \frac12\\
				\frac12 & -\frac7 4
			\end{pmatrix} = -\mathbf{U}^\top \bfmyLambda \mathbf{U}, \quad \text{  with  }~~ \bfmyLambda = \diag\left\{1,\frac{3}{2}\right\} \quand \mathbf{U}= 
			\begin{pmatrix}
		1 & -\frac{1}{2} \\
	0 & 1 \\
	\end{pmatrix}. 
	$$
%	$$
%	\mathbf{B}={\diag}\left\{ 0, -3, -\frac{3}{4} \right\}, \quand \pmb{\Upsilon}=\begin{pmatrix}
%		-1 & \frac12\\
%		\frac12 & -\frac7 4
%	\end{pmatrix}.
%	$$
%	Furthermore, it can be seen that
%	\begin{equation*}
%	\pmb{\Upsilon} = -\mathbf{U}^\top \bfmyLambda \mathbf{U}, \quad \bfmyLambda = \diag\left\{1,\frac{3}{2}\right\},
%	\quand \mathbf{U}= 
%	\begin{pmatrix}
%	1 & -\frac{1}{2} \\
%	0 & 1 \\
%	\end{pmatrix},
%	\end{equation*} 
	By \cref{thm:energy}, we obtain the discrete energy law as
	\begin{equation*}
		\nm{u^{n+1}}^2 - \nm{u^n}^2 = - 3\tau^2\nm{Lw^n}^2 -  \frac{3}{4} \tau^4\nm{L^2w^n}^2 - \tau \qnm{\left(I-\frac{\tau}{2} L\right)w^n}^2- \frac{3}{2}\tau^3 \qnm{L w^n}^2.
	\end{equation*}
\end{example}

\subsection{Examples of conditional stability}
Next, we derive the energy laws for two implicit methods which are not A-stable. Conditional stability can be obtained by \cref{thm:conditionalstability}. 
\begin{example}[Weak stability]
	This example considers the $(0,3)$ Pad\'e approximation with the stability function 
	$
	\mathcal{R}(Z) = \left(I - Z + \frac{Z^2}{2}-\frac{Z^3}{6}\right)^{-1}. 
	$ 
	This method is $A(\alpha)$-stable with $\alpha \leq 88.23^o$; see \cite[Page 46]{wanner1996solving}. If applying it to a generic linear seminegative problem \eqref{eq-odes}, the unconditional stability would not hold in general. According to \cref{lem:abc}, we get 
	\begin{equation*}
	\mathbf{B} = \diag\left\{0,0,\frac{1}{12},-\frac{1}{36}\right\}, \qquad \pmb{\Upsilon} = \left(
	\begin{array}{ccc}
	-1 & \frac{1}{2} & -\frac{1}{6} \\
	\frac{1}{2} & -\frac{1}{3} & \frac{1}{6} \\
	-\frac{1}{6} & \frac{1}{6} & -\frac{1}{12} \\
	\end{array}
	\right).
	\end{equation*}
	Direct calculation shows that $\pmb{\Upsilon}$ has a positive eigenvalue, implying it is not negative semidefinite. But its second-order principle submatrix is negative semidefinite.  Moreover, with 
	$\pmb{\Delta} =\diag\left\{0,0,\frac{1}{36}\right\}$  the matrix 
	$\pmb{\tilde{\Upsilon}} := \pmb{\Upsilon} - \pmb{\Delta}$ is negative semidefinite and admits the following Cholesky type decomposition 
%	\begin{equation*}
%	\pmb{\tilde{\Upsilon}} := \pmb{\Upsilon} - \pmb{\Delta}, \quad \text{where} \quad\pmb{\Delta} =\diag\left\{0,0,\frac{1}{36}\right\},
%	\end{equation*}
%    where $\pmb{\tilde{\Upsilon}}$ is negative semidefinite. 
%	Furthermore, we have 
	\begin{equation*}
	\pmb{\tilde{\Upsilon}} = -\mathbf{\tilde U}^\top \bftildemyLambda\mathbf{\tilde U}, \quad  \bftildemyLambda = \diag\left\{1,\frac{1}{12},0\right\},\quad \mathbf{\tilde{U}} = \left(
	\begin{array}{ccc}
	1 & -\frac{1}{2} & \frac{1}{6} \\
	0 & 1 & -1 \\
	0 & 0 & 1 \\
	\end{array}
	\right).
	\end{equation*}
	Thanks to \cref{thm:energy}, we obtain the energy identity
	{\small \begin{equation*}
	\nm{u^{n+1}}^2 - \nm{u^n}^2 = \frac{\tau^4}{12}\nm{L^2 w^n}^2 -\frac{\tau^6}{36}\nm{L^3 w^n}^2 - {\tau}\qnm{\left(I - \frac{\tau}{2}L +\frac{\tau^2}{6}L^2\right)w^n}^2 - \frac{\tau^3}{12}\qnm{L(I - \tau L)w^n}^2 + \frac{\tau^5}{36}\qnm{L^2 w^n}^2. 
	\end{equation*}}
	Thus $\zeta = 2$ and $\rho = 2$, and by  \cref{thm:conditionalstability}, the $(0,3)$ Pad\'e approximation is weakly$(\kappa)$ stable with $\kappa=4$. 
\end{example}

\begin{example}[Strong stability]
	We consider the $(4,1)$ Pad\'e approximation whose stability function is 
	$
	R(Z) = \left(I - \frac{Z}{5}\right)^{-1}\left(I +\frac{4 Z}{5}+ \frac{3 Z^2}{10}+\frac{Z^3}{15}+\frac{Z^4}{120}\right).
	$ 
According to \cref{lem:abc}, we obtain 
\begin{equation*}
\mathbf{B} = \diag\left\{0,0,0,-\frac{1}{1800},\frac{1}{14400}\right\}\quand \pmb{\Upsilon} = - \left(
\begin{array}{cccc}
1 & \frac{3}{10} & \frac{1}{15} & \frac{1}{120} \\
\frac{3}{10} & \frac{13}{75} & \frac{9}{200} & \frac{1}{150} \\
\frac{1}{15} & \frac{9}{200} & \frac{1}{75} & \frac{1}{400} \\
\frac{1}{120} & \frac{1}{150} &\frac{1}{400} & \frac{1}{1800} \\
\end{array}
\right). 
\end{equation*}
	Direct calculation shows that $\pmb{\Upsilon}$ has a positive eigenvalue, implying it is not negative semidefinite. But its third-order principle submatrix is negative semidefinite.  Moreover, with $\pmb{\Delta} =\diag\left\{0,0,0,\frac{1}{14400}\right\}$, 
	the matrix $\pmb{\tilde{\Upsilon}} := \pmb{\Upsilon} - \pmb{\Delta}$ is negative semidefinite 
	  and admits the following Cholesky type decomposition 
%Note that the eigenvalues of $\pmb{\Upsilon}$ are $\{-1.10298, -0.0827165, -0.0015848, 0.0000555893\}$ and hence $\pmb{\Upsilon}$ is not negative semidefinite. But its third-order principle submatrix is negative semidefinite. Note that $\pmb{\tilde{\Upsilon}} := \pmb{\Upsilon} - \pmb{\Delta}$ is negative semidefinite 
%with $\pmb{\Delta} =\diag\left\{0,0,0,\frac{1}{14400}\right\}.$  
%\begin{equation*}
%\pmb{\tilde{\Upsilon}} := \pmb{\Upsilon} - \pmb{\Delta}\leq 0, \quad \text{where} \quad\pmb{\Delta} =\diag\left\{0,0,0,\frac{1}{14400}\right\}.
%\end{equation*}
%Furthermore, we have 
\begin{equation*}
\pmb{\tilde{\Upsilon}} = -\mathbf{\tilde U}^\top \bftildemyLambda\mathbf{\tilde U}, \quad  \bftildemyLambda = \diag\left\{1,\frac{1}{12},\frac{1}{720},0\right\},\quad \mathbf{\tilde{U}} = \left(
\begin{array}{cccc}
1 & \frac{3}{10} & \frac{1}{15} & \frac{1}{120} \\
0 & 1 & \frac{3}{10} & \frac{1}{20} \\
0 & 0 & 1 & \frac{1}{2} \\
0 & 0 & 0 & 1 \\
\end{array}
\right).
\end{equation*}
According to \cref{thm:energy}, we have the following energy law
\begin{align*}
\nm{u^{n+1}}^2 - \nm{u^n}^2 
=&-\frac{\tau^6}{1800}\nm{L^3w^n}^2 + \frac{\tau^8}{14400}\nm{L^4 w^n}^2
-\tau\qnm{\left(I + \frac{3\tau}{10}L + \frac{\tau^2}{15}L^2 + \frac{\tau^3}{120}L^3\right)w^n}^2
\\ 
& -\frac{\tau^3}{12}\qnm{L\left(I+\frac{3\tau}{10}L + \frac{\tau^2}{20}L^2\right)w^n}^2
-\frac{\tau^5}{720}\qnm{L^2\left(I+\frac{\tau}{2}L\right)w^n}^2 + \frac{\tau^7}{14400}\qnm{L^3 w^n}^2. 
\end{align*}
This implies $\zeta = 3$, $\rho = 3$, and $\beta_\zeta <0$. We conclude the conditional strong stability  from \cref{thm:conditionalstability}. 
%under some time step constraint $\tau \nm{L}\leq \mycfl_0$. 
\end{example}

\section{Unified energy law for general diagonal Pad\'e approximations}\label{sec:pade} 
In this section, we derive the unified discrete energy law for general diagonal Pad\'e approximations of arbitrary order. {\em The establishment of such energy law will be based on highly technical Cholesky type decomposition of a family of complicated matrices, 
whose discovery and proof are extremely nontrivial.} 

For the $(s,s)$ diagonal Pad\'e approximation, the stability function is given by \eqref{eq:stabilityfunc} and \eqref{eq:ndpoly} with the coefficients in  \eqref{eq:ndpoly} defined as
\begin{equation}\label{eq:mybeta}
	\theta_i = (-1)^i\vartheta_i = \frac{s!}{(2s)!}\frac{(2s-i)!}{i!(s-i)!}.
\end{equation}
Thus we have $\alpha_{i,j} = \theta_i \theta_j - \vartheta_{i} \vartheta_{j} = ( 1 - (-1)^{i+j} ) \theta_i \theta_j$.
According to \cref{lem:abc}, the matrix $\mathbf{B}=\diag(\{\beta_k\}_{k=0}^{s}) = {\bf O}$  because
\begin{equation*} 
	\beta_k = \sum_{\ell = \max\{0,2k-s\}}^{\min\{2k,s\}} \alpha_{\ell,2k-\ell}(-1)^{k-\ell} = \sum_{\ell = \max\{0,2k-s\}}^{\min\{2k,s\}}  ( 1-(-1)^{2k} )  \theta_{\ell} \theta_{2k-\ell}  (-1)^{k-\ell} = 0, 
\end{equation*}
and the symmetric matrix $\pmb{\Upsilon}=(\gamma_{i,j})_{i,j=0}^{s-1}$ is computed by 
\begin{align} \nonumber
	\gamma_{i,j} %&= \sum_{\ell = \max\{0,i+j+1-s\}}^{\min\{i,j\}} (-1)^{\min\{i,j\}+1-\ell}\alpha_{\ell,i+j+1-\ell}
	%\\ 
	&= \sum_{\ell = \max\{0,i+j+1-s\}}^{\min\{i,j\}} (-1)^{\min\{i,j\}+1-\ell} \Big( 1 - (-1)^{i+j+1} \Big) \theta_{\ell} \theta_{i+j+1-\ell}
	\\ \label{eq:gammaPade1}
	& = \Big( (-1)^i + (-1)^j \Big) \sum_{\ell = \max\{0,i+j+1-s\}}^{\min\{i,j\}} (-1)^{\ell + 1} \theta_{\ell} \theta_{i+j+1-\ell} 
	\\ \label{eq:gammaPade2}
	& = \Big(  (-1)^i + (-1)^j \Big) \left( \frac{s!}{(2s)!} \right)^2 \sum_{\ell = \max\{0,i+j+1-s\}}^{\min\{i,j\}} (-1)^{\ell + 1} 
	 \frac{(2s-\ell )!}{\ell !(s-\ell )!} 
	\frac{(2s-i-j-1+\ell)!}{(i+j+1-\ell)!(s-i-j-1+\ell)!}.
\end{align}

In order to establish the energy identity, the key step is to judge the negative semi-definiteness of the above matrix $\pmb{\Upsilon}$ and construct its Cholesky type decomposition. 
For an arbitrary $s \in \mathbb Z^+$, this is indeed a highly challenging task, because the structures of $\pmb{\Upsilon}$ are extremely complicated and all its elements \cref{eq:gammaPade2} involve complex summations of several factorial products.

{\bf After careful investigation, we find the unified explicit form of the Cholesky type decomposition of $\pmb{\Upsilon}$, as stated in \cref{thm:pade-cholesky}.}   

% via \cref{thm:energy}
%\subsection{Constructive matrix decomposition}

%From \cref{lem:abc}, it can be seen that 
%\[
%\mathbf{B} = \mathbf{O}
%\]
%is a zero matrix. Furthermore, we have the following theorem on $\pmb{\Upsilon}$. 
%the $(s,s)$ diagonal Pad\'e approximation

\begin{theorem}[Constructive matrix decomposition]\label{thm:pade-cholesky}For any $s \in \mathbb Z^+$, the symmetric matrix $\pmb{\Upsilon}$ defined by \cref{eq:gammaPade2} is always negative definite. Furthermore, it has the Cholesky type decomposition in the following unified explicit form
	\begin{equation}\label{eq:pade-cholesky}
	\pmb{\Upsilon}  = - \mathbf{U}^\top \bfhatmyLambda \mathbf{U},
	\end{equation}
	where $\bfhatmyLambda = \diag(\{ \widehat{\mylambda}_k \}_{k=0}^{s-1})$ with $\widehat{\mylambda}_k =  \frac{(k!)^2}{(2k)!(2k+1)!}$, and 
	$\mathbf{U} = (\mu_{i,j})_{i,j=0}^{s-1}$ is an upper triangular matrix with 
	\begin{equation}\label{eq:mu-pade}
	\mu_{i,j} := 
	\begin{cases}
	\displaystyle
	\frac{s!}{(2s)!}\frac{(2i+1)!}{i!(i+j+1)!}\frac{(2s+i-j)!}{(s-1-j)!}\frac{(s-1-\frac{i+j}{2})!(\frac{i+j}{2})!}{(s-\frac{j-i}{2})!(\frac{j-i}{2})!}, & \mbox{if}~ i \le j~{\mathrm{and}}~i \equiv j~({\mathrm {mod}}~2), 
	\\
	0, & \mbox{otherwise}.
	\end{cases}
	\end{equation}
\end{theorem}

%The discovery of the Cholesky type decomposition in \cref{thm:pade-cholesky} is very nontrivial. 
The proof of \cref{thm:pade-cholesky} is very technical and will be given in \cref{sec:choleskyPade} for better readability.

\subsection{Unified discrete energy law and unconditional stability}

Combining \cref{thm:pade-cholesky} with \cref{thm:energy}, we immediately obtain the discrete energy laws of 
all the diagonal Pad\'e approximations in 
a unified form.  

\begin{theorem}[Unified energy law and unconditional stability]\label{thm:pade-ed-series}
	For any $s \in \mathbb Z^+$, 
	 the $(s,s)$ diagonal Pad\'e approximation for general linear seminegative system \eqref{eq-odes} admits the following discrete energy law
	\begin{equation}\label{eq:pade-ed-series}
	\nm{u^{n+1}}^2 - \nm{u^n}^2 = -\sum_{k = 0}^{s-1} \widehat{\mylambda}_k \tau^{2k+1}\qnm{L^k {u}^{(k)}}^2,
	\end{equation}
	where $\widehat{\mylambda}_k =  \frac{(k!)^2}{(2k)!(2k+1)!}$ and 
	\begin{equation}\label{EL-d-u}
	{u}^{(k)} := \sum_{j = k }^{s-1} {\mu}_{k,j}(\tau L)^{j-k} Q^{-1} u^n,
	\end{equation}
	with ${\mu}_{k,j}$ defined by \eqref{eq:mu-pade}. The energy law \cref{eq:pade-ed-series} implies 
	$\nm{u^{n+1}} \le \nm{u^n}$ for all $\tau >0$, which means all diagonal Pad\'e approximations 
	are  unconditionally strongly stable for general linear seminegative systems. 
\end{theorem}

\subsection{Connections between continuous and discrete energy laws}\label{sec:connection}

Having found the above unified discrete energy law, % for all the diagonal Pad\'e approximations, 
we are now in the position to explore the connections between the continuous energy law \cref{eq:ode-ed-series} in 
\cref{thm:ode-ed-series} and the discrete energy law \cref{eq:pade-ed-series} in \cref{thm:pade-ed-series}. 

In fact, the discrete energy law \cref{eq:pade-ed-series} of the $(s,s)$ diagonal Pad\'e approximation 
is a truncated approximation to the continuous energy law \cref{eq:ode-ed-series}. 
It is clearly seen that the continuous and discrete laws share the same 
 expansion coefficients $\widehat{d}_k$ of the first $s$ terms. 
Although the quantity $u^{(k)}$ in \cref{EL-d-u} is not exactly equal  
to $\widehat{u}^{(k)}$ in \cref{eq:uhatk}, 
they actually match up to high order. 
Notice that the series
$u^{(k)}$ in \cref{EL-d-u} is expanded in terms of $w^n=Q^{-1} u^n$, while  $\widehat{u}^{(k)}$ in \cref{eq:uhatk} is expanded in terms of $u(t^n)$. 
For ease of comparison, we can either reformulate $\widehat{u}^{(k)}$ in the similar form as $u^{(k)}$ (see \cref{lem:uhatu}), or rewrite $u^{(k)}$ in the similar form as $\widehat{u}^{(k)}$ (see \cref{lem:uuhat}). 
In order to rigorously show these two theorems, 
we need the important combinatorial identity in \cref{lem:uuhat-cid}, whose proof is  provided in \cref{proof:uuhat-cid}.

%An important combinatorial identity is needed in our proof. 

%\subsection{A combinatorial identity}
\begin{lemma}\label{lem:uuhat-cid}
	For any $i,j\in \mathbb N$ and $s \in \mathbb Z^+$ with $0 \le i \le j \leq s-1$, it holds that 
	\begin{equation*}%\label{e:uuhat-cid}
		%\begin{aligned}
			\sum_{\kk = 0}^{j-i} \binom{\mys-\kk}{j-\kk}^{-1}\binom{2\mys-\kk}{j-i-\kk}\binom{i+j+1}{\kk}(-1)^\kk\\
			=
			\begin{cases}
				\displaystyle
				\frac{(\mys-1-\frac{i+j}{2})!(\frac{i+j}{2})!}{(\mys-\frac{j-i}{2})!(\frac{j-i}{2})!}(s-j), & {\rm if}~i \le j~{\mathrm{and}}~i \equiv j~({\mathrm {mod}}~2), 
				\\
				0, & {\rm otherwise}.
			\end{cases}
		%\end{aligned}
	\end{equation*}
\end{lemma}

\begin{theorem}\label{lem:uhatu}
	Suppose $u^n = u(t^n)$. 
	The series $\widehat{u}^{(k)}$ in \cref{eq:uhatk} can be equivalently rewritten as 
	\begin{equation}
	\widehat{u}^{(k)} = \sum_{j = k}^\infty \bar{\mu}_{k,j}(\tau L)^{j-k} Q^{-1}u^n,
	\end{equation}
	where $\bar{\mu}_{k,j} := \sum_{\ell=\max\{j-s,k\}}^j\widehat{\mu}_{k,\ell}\vartheta_{j-\ell}$. Moreover, 
	the coefficients 
	$\bar{\mu}_{k,j}$ exactly coincide with those in \cref{EL-d-u}, namely, 
	 $\bar{\mu}_{k,j} = \mu_{k,j}$ for $k\leq j \leq s-1$. 
\end{theorem}

\begin{proof}
	Substituting $u^n=Q w^n = \sum_{k = 0}^\mys \vartheta_k (\tau L)^k w^n$ into \eqref{eq:uhatk}, we obtain 
	\begin{align*}
		\widehat{u}^{(i)} &=  \sum_{\ell = i}^\infty \widehat{\mu}_{i,\ell}(\tau L)^\ell \left(\sum_{k = 0}^\mys \vartheta_k (\tau L)^k \right) w^n
		= \sum_{\ell = i}^\infty\sum_{k = 0}^\mys \widehat{\mu}_{i,\ell} \vartheta_k (\tau L)^{\ell+k} w^n
		\\
		&
		= \sum_{\ell = i}^\infty\sum_{j = \ell}^{\ell+\mys} \widehat{\mu}_{i,\ell} \vartheta_{j-\ell} (\tau L)^j w^n
		= \sum_{j = i}^\infty\left(\sum_{\ell = \max\{j-\mys,i\}}^j \widehat{\mu}_{i,\ell} \vartheta_{j-\ell} \right)(\tau L)^j w^n=: \sum_{j = i}^\infty\bar{\mu}_{i,j} (\tau L)^j w^n.
	\end{align*}
	Recall the definitions of $\widehat{\mu}_{i,j}$ and $\vartheta_i$ in \eqref{eq:lambdamu_hat} and \eqref{eq:mybeta}, respectively. Substituting them into $\bar{\mu}_{i,j}$, we have
	\begin{align*}%\label{eq:mubar}
		%\begin{aligned}
		\bar{\mu}_{i,j}=& \sum_{\kk = \max\{j-\mys,i\}}^j \frac{(2i+1)! \kk!}{i!(\kk-i)!(\kk+i+1)!} \frac{\mys!}{(2\mys)!}\frac{(2\mys-(j-\kk))!}{(j-\kk)!(\mys-(j-\kk))!}(-1)^{j-\kk}\\
		%=& \frac{\mys!}{(2\mys)!}\frac{(2i+1)!}{i!}\sum_{\kk = \max\{j-\mys,i\}}^j \frac{ \kk!}{( \kk-i)!( \kk+i+1)!} \frac{(2\mys-(j-\kk))!}{(j-\kk)!(\mys-(j-\kk))!}(-1)^{j-\kk}\\
		=& \frac{\mys!}{(2\mys)!}\frac{(2i+1)!}{i!}\frac{(2\mys+i-j)!}{(\mys-j)!}\sum_{\kk = \max\{j-\mys,i\}}^j \frac{  (-1)^{j-\kk}  \kk!}{(\kk-i)!(\kk+i+1)!}\frac{(\mys-j)!}{(2\mys+i-j)!} \frac{(2\mys-(j-\kk))!}{(j-\kk)!(\mys-(j-\kk))!}\\
		=& \frac{\mys!}{(2\mys)!}\frac{(2i+1)!}{i!(i+j+1)!}\frac{(2\mys+i-j)!}{(\mys-j)!}\sum_{\kk = \max\{j-\mys,i\}}^j (-1)^{j-\kk} \binom{\mys+\kk-j}{\kk}^{-1}\binom{2\mys+\kk-j}{\kk-i}\binom{i+j+1}{j-\kk}\\
		=& \frac{\mys!}{(2\mys)!}\frac{(2i+1)!}{i!(i+j+1)!}\frac{(2\mys+i-j)!}{(\mys-j)!}\sum_{\kk = 0}^{\min\{j-i,\mys\}} \binom{\mys-\kk}{j-\kk}^{-1}\binom{2\mys-\kk}{j-i-\kk}\binom{i+j+1}{\kk}(-1)^\kk.
		%\end{aligned}
	\end{align*}
	Note that when $i \leq j \leq s-1$, we have $\min\{j-i,s\} = j - i$. 
	Using the combinatorial identity in \cref{lem:uuhat-cid}, we obtain 
	 $\bar{\mu}_{i,j} = \mu_{i,j}$ for $0\leq i, j\leq s-1$. The proof is completed. 
\end{proof}

\begin{theorem}\label{lem:uuhat}
	The series ${u}^{(k)}$ in \cref{EL-d-u} can be equivalently reformulated as 
	\begin{equation*}
	{u}^{(k)} = \sum_{j = k }^{s-1} \widehat{\mu}_{k,j} (\tau L)^{j-k} I_j u^n, %\quad ,
	\end{equation*}
	where $I_j := Q_jQ^{-1}$ with $Q_j := \sum_{i=0}^{s-1-j}\vartheta_{i}(\tau L)^i$ denoting the $(s-1-j)$th order truncation of $Q$.
\end{theorem}

\begin{proof}

	According to \cref{lem:uhatu}, we have $\bar{\mu}_{i,j} = \mu_{i,j}$ for $0\leq i,j\leq s-1$. In this case, $\max\{j-s,i\} = i$ and thus 
	%the computation in \eqref{eq:mubar} gives ${\mu}_{i,j} = \sum_{\kk=i}^j\widehat{\mu}_{i,\kk}\vartheta_{j-\kk}$. 
	$\mu_{i,j} = \bar{\mu}_{i,j} = \sum_{\ell=\max\{j-s,i\}}^j\widehat{\mu}_{i,\ell}\vartheta_{j-\ell} = \sum_{\ell=i}^j\widehat{\mu}_{i,\ell}\vartheta_{j-\ell} $. 
	Substituting this into \cref{EL-d-u} gives 
	\begin{align*}
		u^{(i)} 
		=& \sum_{j = i }^{s-1} \left(\sum_{\kk=i}^j\widehat{\mu}_{i,\kk}\vartheta_{j-\kk}\right)(\tau L)^{j-i} Q^{-1} u^n
		=\sum_{\kk=i}^{s-1}\sum_{j=\kk}^{s-1}\widehat{\mu}_{i,\kk}\vartheta_{j-\kk} (\tau L)^{j-i}Q^{-1}u^n\\
		=&\sum_{\kk=i}^{s-1}\sum_{j=0}^{s-1-\kk}\widehat{\mu}_{i,\kk}\vartheta_{j} (\tau L)^{j+\kk-i}Q^{-1}u^n
		=\sum_{\kk=i}^{s-1}\widehat{\mu}_{i,\kk}(\tau L)^{\kk-i}\left(\sum_{j=0}^{s-1-\kk}\vartheta_{j} (\tau L)^j\right)Q^{-1}u^n,
	\end{align*}
	which completes the proof.
\end{proof}

\begin{remark}\label{rem:DE_accuracy}
\cref{lem:uhatu} together with \cref{thm:ode-ed-series} and \cref{thm:pade-ed-series} gives the following estimation of the accuracy of the energy dissipation
%For a fixed $L$, we have
\begin{equation*}
	\left(\nm{u(t^{n+1})}^2 - \nm{u(t^n)}^2\right) - \left(\nm{u^{n+1}}^2 - \nm{u^n}^2\right)  = \mathcal{O}(\tau^{2s+1}),
\end{equation*}
which implies for a fixed $T=n\tau$ that the total energy dissipation accuracy $\Delta E:=(\nm{u(t^{n})}^2-\nm{u(t^0)}^2) - (\nm{u^{n}}^2 - \nm{u^0}^2) =  \mathcal{O}(\tau^{2s})$. 
\end{remark}

%\cref{lem:uhatu} together with \cref{thm:ode-ed-series} and \cref{thm:pade-ed-series} gives the following estimation of the accuracy of the energy dissipation. The detailed proof is omitted. 
%\begin{corollary}\label{thm:energy-accuracy} For a fixed $L$, we have
%	\begin{equation}
%	\left(\nm{u(t^{n+1})}^2 - \nm{u(t^n)}^2\right) - \left(\nm{u^{n+1}}^2 - \nm{u^n}^2\right)  = \mathcal{O}(\tau^{2s+1}).
%	\end{equation}
%\end{corollary}

\begin{remark}
Combining \cref{thm:pade-ed-series} with \cref{lem:uuhat}, we can derive the following precise characterization 
on the operator $\mathcal{R}(\tau L)$: 
	\begin{equation}\label{eq:ddd}
	(\mathcal{R}(\tau L))^\top \mathcal{R}(\tau L) -I=  \sum_{k = 0}^{s-1}\widehat{\mylambda}_k \tau^{2k+1} {U}_k^\top(L^\top+L){U}_k\leq O, \quad \mbox{with}~~ {U}_k = L^k\sum_{j = k}^{s-1} \widehat{\mu}_{k,j}(\tau L)^{j-k}I_j,
\end{equation}
where $\widehat{\mylambda}_{k}$ and $\widehat{\mu}_{k,j}$ are defined in \eqref{eq:lambdamu_hat}, and $I_j$ is defined in \cref{lem:uuhat}.  
Note that the operator $\mathcal{R}(\tau L)$ is the discrete approximation to the operator $e^{\tau L}$. 
The identity \cref{eq:ddd} on $\mathcal{R}(\tau L)$ is exactly the discrete counterpart of the identity \cref{eq:coercivity} on $e^{\tau L}$ of the continuous case. 
\end{remark}

In summary, our above analyses clearly demonstrate the 
unity of continuous and discrete objects.

%the unity of opposites between continuous and discrete.
%correspondence and 

%counterpart

% implies the following characterization on the weak coercivity of the RK operator, which can be compared with \cref{cor:coercivity}. 
%\begin{corollary}
%	\begin{equation*}
%	(\mathcal{R}(\tau L))^\top \mathcal{R}(\tau L) = I - \tau\sum_{k = 0}^{s-1}\widehat{\mylambda}_k {U}_k^\top(L^\top+L){U}_k, \quad {U}_k = \sum_{j = k}^{s-1} \widehat{\mu}_{k,j}(\tau L)^jI_j,
%	\end{equation*}
%	where $\widehat{\mylambda}_{k}$, $\widehat{\mu}_{k,j}$ and $I_j$ are defined in \eqref{eq:lambdamu_hat} and \cref{lem:uuhat}, respectively. 
%\end{corollary}

\subsection{Proof of \cref{thm:pade-cholesky}}\label{sec:choleskyPade}
The discovery and proof of \cref{thm:pade-cholesky} are highly nontrivial and challenging. Our proof is very technical and relies on several lemmas and constructive identities.

Note that the negative definiteness of $\pmb{\Upsilon}$ is implied by the existence of the Cholesky type decomposition \cref{eq:pade-cholesky} with positive $\widehat d_k$ for all $k$. 
Therefore, we only need to prove the identity \cref{eq:pade-cholesky} for any $s \in \mathbb Z^+$. 
Define $\mathbf{F}(s) := \pmb{\Upsilon} + \mathbf{U}^\top \bfhatmyLambda\mathbf{U}$. Then the goal is to show that the matrix-valued function $\mathbf{F}(s)\equiv \mathbf{O}$ is identically zero for all $s\in \mathbb{Z}^+$. 

Let ${\mathcal F}_{p,q}(s)$ denote the $(p,q)$ element of ${\bf F}(s)$. 
In order to clearly show the dependence of ${\mathcal F}_{p,q}(s)$ on $s$, we will equivalently  
reformulate it with some new notations. First, we introduce 
\begin{equation}\label{eq:Def-beta}
	\mybeta_0^{(\mys)} := 1, \qquad 
	\mybeta_i^{(\mys)} :=
	\frac{1}{i!} ~ \frac{ \mys(\mys-1) \cdots (\mys-i+1) }{ { 2\mys( 2\mys-1) \cdots ( 2\mys-i+1) } }, \quad i\in \mathbb Z^+,
\end{equation}
which satisfy $\mybeta_i^{(\mys)} = \theta_i$ for $0\le i \le \mys$, and $\mybeta_i^{(\mys)} = 0$ for $\mys<i<2 \mys$. Furthermore, we define 
\begin{equation}\label{eq:Defg}
	\gamma_{p,q}^{(\mys)} := \left[ (-1)^p + (-1)^q \right] \sum_{i=0}^{\min \{p,q\}} (-1)^{i+1} \mybeta_i^{(\mys)} \mybeta_{p+q+1-i}^{(\mys)}, \qquad p,q\in \mathbb N.
\end{equation}
Note that $\mybeta_{p+q+1-i}^{(\mys)}=0$ for $p+q+1-i>s$, which along with \cref{eq:gammaPade1} implies 
\begin{equation}\label{eq:WKL3212}
	\gamma_{p,q}^{(s)} = \gamma_{p,q},\quad 0\leq p,q\leq s-1. 
\end{equation}
For $i,j \in \mathbb Z^+$, we define
\begin{equation}\label{eq:Defalpha}
	\myalpha_{i,j}^{(\mys)} := 
	\begin{cases}
		\displaystyle
		\frac{ \mys! }{ (2\mys)! } \frac{ 2 \sqrt{2i-1} }{ (i+j)! } 
		\frac{ (2\mys+i-j)! \left(\mys - \frac{i+j}2 \right)! \left(  \frac{i+j}2 \right)!  } 
		{ (\mys-j)! \left( \mys - \frac{j-i}2 \right)! \left(  \frac{j-i}2 \right)!  }, & \mbox{if}~ i \le j~{\mathrm{and}}~i \equiv j~({\mathrm {mod}}~2), 
		\\
		0, & \mbox{otherwise}.
	\end{cases}
\end{equation}
One can verify that $\myalpha_{i,j}^{(s)} = \sqrt{\mylambda_{i-1}}\mu_{i-1,j-1}$ for $1\leq i,j\leq s$. 
Therefore, for $1\le p,q\le s$, ${\mathcal F}_{p,q}(s)$ can be equivalently reformulated as 
\begin{equation}\label{eq:DefFpq}
	{\mathcal F}_{p,q}(s) = \gamma_{p-1,q-1}^{ (\mys) } + \sum_{i=1}^{\min\{p,q\}} \myalpha_{i,p}^{(\mys)} \myalpha_{i,q}^{(\mys)} \qquad \forall  s \in \mathbb Z %\qquad \forall  1\le p,q\le s,  \quad . 
\end{equation}

We have the following two crucial observations. 

\begin{observation}\label{OB1}
	For any fixed $p,q \in \mathbb Z^+$, the function ${\mathcal F}_{p,q}(s)$ in \cref{eq:DefFpq} is a rational function of $s$.  
\end{observation}

\begin{proof}
For any fixed $i \in \mathbb N$, the function $\mybeta_i^{(\mys)}$ defined in \cref{eq:Def-beta} is a rational function of $s$, and thus for any fixed $p,q\in \mathbb N$, the function $\gamma_{p,q}^{(\mys)}$ is also a rational function of $s$. 
Note that for any fixed $i,j \in \mathbb Z^+$, $\myalpha_{i,j}^{(\mys)}$ in \cref{eq:Defalpha} can be 
easily rewritten as a rational function of $s$. 
Therefore, for any fixed $p,q\in \mathbb Z^+$, all the terms in \cref{eq:DefFpq} are rational functions of $s$, and thus ${\mathcal F}_{p,q}(s)$ is also a rational function of $s$.  
\end{proof}

\begin{observation}\label{OB2}
All elements of $\mathbf{F}(s)$ are rational functions of $s$. 
Recall that 
a rational function vanishes at only finite points unless it is identically zero. 
Therefore, if we can prove that all elements ${\mathcal F}_{p,q}(s)$ vanish for all $s$ on an {\bf uncountable} set $\widehat{ \mathbb R }$, then it forces $\mathbf{F}(s)\equiv \mathbf{O}$ for all $s \in \mathbb{Z}^+$. 
\end{observation}

For convenience, hereafter the factorial is extended to represent the gamma function $\Gamma(x+1)$, namely,  
$$ 
x! ~ := \Gamma (x+1), \qquad \forall x \in \mathbb R \setminus \mathbb Z^-.
$$
In our following lemmas and proofs, we will introduce some intermediate quantities that are also rational functions of $\mys$, whose denominators may vanish at $\left\{ 0,\pm \frac12, \pm 1, \pm \frac32, \dots \right\}$. To avoid potential singularity of dividing a zero denominator, we will extend the domain of $s$ from $\mathbb Z^+$ to $\mathbb{R}$ but excluding all potential singular points. 
More specifically, we will prove the following proposition. 
\begin{proposition}\label{prop:key-cholesky}
	For all $p,q \in \mathbb Z^+$, the rational function $F_{p,q}(s)$ vanishes for all $s\in \widehat{ \mathbb R }$, namely, 
		\begin{equation}\label{eq:key-cholesky}
		\gamma_{p-1,q-1}^{ (\mys) } + \sum_{i=1}^{\min\{p,q\}} \myalpha_{i,p}^{(\mys)} \myalpha_{i,q}^{(\mys)} 
		= 0 \qquad \forall p,q \in \mathbb Z^+, \qquad \forall s\in \widehat{ \mathbb R }, 
	\end{equation}
	where 
	\begin{equation}\label{eq:hatR}
		\widehat{\mathbb R}:=
		\left\{ x \in \mathbb R:~ 2x \notin \mathbb Z \right\} = \mathbb R \setminus \left\{ 0,\pm \frac12, \pm 1, \pm \frac32, \dots \right\}.
	\end{equation}
	%Note that the summation in \cref{eq:key-cholesky} is actually {\bf finite}, since $\myalpha_{i,j} = 0$ as $j>i$. The  
	%\gamma_{p-1,q-1}^{ (\mys) } + \sum_{i=1}^{\infty} \myalpha_{i,p}^{(\mys)} \myalpha_{i,q}^{(\mys)} 
	%= 0
\end{proposition}

The proof of \cref{prop:key-cholesky} relies on several lemmas in \cref{sec:lemmas}
and will be given in \cref{sec:cholesky-proof}. 
Note that the set $\widehat{\mathbb R}$ defined in \cref{eq:hatR} is uncountable. 
Based on \cref{OB1}, \cref{OB2} and the above arguments, once we 
 prove \cref{prop:key-cholesky}, then we immediately obtain 
\cref{eq:pade-cholesky} for all $s \in \mathbb Z^+$ and complete the proof of \cref{thm:pade-cholesky}.

\subsection{Lemmas}\label{sec:lemmas}
This section gives several important lemmas, which pave the way to prove \cref{prop:key-cholesky}. 
First, we introduce the rising factorial (sometime also called the Pochhammer symbol in the theory of hypergeometric functions),  defined by 
\begin{equation}\label{eq:Def_Pochhammer}
	(x)_0 := 1, \qquad	(x)_n :=  x (x + 1) \cdots (x+n-1) = \prod\limits_{k=0}^{n-1}  (x+k),~~~ n \ \in \mathbb Z^+, 
\end{equation}
for any $x\in \mathbb R$. Note that 
$$(x)_n \neq 0\qquad \forall x \notin \mathbb Z \quad \forall n \in \mathbb N.$$

\cref{lem:identities} gives three useful identities related to the Pochhammer symbol, whose proofs are presented in \cref{sec:identities}. 
\begin{lemma}\label{lem:identities}
	The following identities hold: 
	\begin{align}\label{ID-1}
		(x+n)! & = x! (x+1)_n \qquad \qquad  \qquad \qquad 
		\forall x \in \mathbb R,~~ \forall n \in \mathbb N,
		\\ \label{ID-2}
		(x)_n & = 2^n \left( \frac{x}2 \right)_{ \left \lceil{ \frac{n}2 }\right \rceil  } \left( \frac{x+1}2 \right)_{ \left \lfloor{ \frac{n}2 }\right \rfloor  }  \qquad \forall x \in \mathbb R,~~ \forall n \in \mathbb N,
		\\ \label{ID-3}
		\frac{ (x+i)! }{ (x-j)! } &= (-1)^j (-x)_j ( x+1 )_i \qquad  \quad ~~
		\forall x \in \mathbb R \setminus \{ j-1,j-2,\dots \},~~ \forall  i,j \in \mathbb N.
	\end{align}
\end{lemma}

%which give \eqref{ID-1} and \eqref{ID-2}. 
%It follows from \eqref{ID-1} that 
%\begin{equation*}
%	(x+i)!  = x! ( x+1 )_i  = (x-j)!    (x-j+1)(x-j+2) \cdots (x-1) x \cdot ( x+1 )_i 
%	= (x-j)!   (-1)^j (-x)_j (x+1)_i, 
%\end{equation*}
%which implies \eqref{ID-3}. 

Note for any fixed $i,j \in \mathbb Z^+$ that  $\myalpha_{i,j}^{(\mys)}$ is also a rational function of $\mys$. We now establish the relations between $\myalpha_{i,j}^{(\mys)}$ and $\mybeta_{j}^{(\mys)}$.  
\begin{lemma}\label{lem:alpha-beta}
	For any $i,j \in \mathbb Z^+$ and any $\mys \in \widehat {\mathbb R}$, we have 
	\begin{align}\label{eq:AB1}
		\myalpha_{2i,2j}^{(\mys)} & = 2 \sqrt{4i-1} 
		\frac{ \left( \mys+\frac12 - j \right)_i (-j)_i }{ (j-\mys)_i \left( \frac12 +j \right)_i } \mybeta_{2j}^{ (\mys) },
		\\[2mm] 
		\myalpha_{2i-1,2j-1}^{(\mys)} &= 2 \sqrt{4 i -3} 
		\frac{ \left( \mys +\frac32 - j  \right)_{i-1} (1-j)_{i-1} }
		{ (j-\mys)_{i-1} \left( \frac12 +j \right)_{i-1} } \mybeta_{2j-1}^{ (\mys) }.
		\label{eq:AB2} 
	\end{align}
\end{lemma}
The proof of \cref{lem:alpha-beta} is put in \cref{sec:alpha-beta}.

For $p,q \in \mathbb Z^+$, define the following two sequences of rational functions of $\mys$: for $n=0,1,\dots$,  
\begin{align}\label{Def_varphi}
	\varphi_{n} (\mys;p,q) & := \frac{  \left(  \mys + \frac32  - p  \right)_n (1-p)_n  \left(  \mys+\frac12 - q \right)_n (-q)_n }
	{  (p-\mys+1)_n  \left(   p + \frac32 \right)_n (q-\mys+1)_n
		\left(  q + \frac32 \right)_n   }, 
	\\ \label{Def_phi}
	\phi_n (\mys;p,q) & := \varphi_{n} (\mys;p,q)  
	\frac{  {\mathcal C}_{n,p,q}^{1,\mys} + {\mathcal C}_{n,p,q}^{2,\mys}  }
	{ (\mys-p)(1+2p)(\mys-q)(1+2q) } 
	%	\Bigg(  
	%	\frac{   }
	%	{ (m-p)(1+2p)(m-q)(1+2q) }
	%	\\
	%	& \qquad \qquad \qquad +\frac{    }
	%	{ (m-p)(1+2p)(m-q)(1+2q) }
	%	\Bigg).
\end{align}
with 
\begin{align*}
	{\mathcal C}_{n,p,q}^{1,\mys} &:=(4n+3)(1+\mys-2p)(q-n)(1+2\mys+2n-2q),
	\\
	{\mathcal C}_{n,p,q}^{2,\mys} &:= (4n+1)(\mys-2q)(1+2p+2n)(\mys-p-n).
\end{align*}
Notice that for all $n\ge p$, we have $(1-p)_n=0$, so that 
\begin{equation}\label{eq:ngep}
	\varphi_{n} (\mys;p,q)  = 0, \quad \phi_n (\mys;p,q) =0 \qquad \forall n\ge p. 
\end{equation}

\begin{lemma}\label{lem:alpha_beta_sigma}
	For any $\mys \in \widehat {\mathbb R}$, it holds 
	\begin{equation}\label{eq:alpha_beta_sigma}
		\myalpha_{2i-1,2p-1}^{(\mys)}  \myalpha_{2i-1,2q+1}^{(\mys)} 
		+ \myalpha_{2i,2p}^{(\mys)}  \myalpha_{2i,2q}^{(\mys)} = 2 \mybeta_{2p-1}^{(\mys)} \mybeta_{2q}^{(\mys)} \phi_{i-1}(\mys;p,q) \qquad 
		\forall i,p,q \in \mathbb Z^+. 
	\end{equation}
\end{lemma}

\begin{proof}
	Denote $n=i-1$. 
	Using Lemma \ref{lem:alpha-beta} gives 
	\begin{align*}
		\myalpha_{2i-1,2p-1}^{(\mys)}  \myalpha_{2i-1,2q+1}^{(\mys)} 
		& \stackrel{\eqref{eq:AB2}}{=} 4 (4i-3) \mybeta_{2p-1}^{(\mys)} \mybeta_{2q+1}^{(\mys)} 
		\frac{ \left( \mys+\frac32-p \right)_{i-1} (1-p)_{i-1} }
		{ (p-\mys)_{i-1}\left( p+\frac12 \right)_{i-1} }
		\frac{ \left( \mys+\frac12-q \right)_{i-1} (-q)_{i-1} }
		{ (q+1-s)_{i-1}\left( q+\frac32 \right)_{i-1} }
		\\
		& = 4 (4n+1) \mybeta_{2p-1}^{(\mys)} \mybeta_{2q+1}^{(\mys)}  \frac{ (p+1-\mys)_{n}\left( p+\frac32 \right)_{n} }{ (p-\mys)_{n}\left( p+\frac12 \right)_{n} } \varphi_{n} (\mys;p,q)
		\\
		& = 4 (4n+1)  \mybeta_{2p-1}^{(\mys)} 
		\frac{  \mybeta_{2q}^{(\mys)}  ( \mys-2q) }{ (2\mys-2q)(2q+1) }  \frac{ (p-\mys+n) \left(p+n+\frac12\right) }{ (p-\mys)\left( p+\frac12 \right) } \varphi_{n} (\mys;p,q)
		\\
		& = 2 \mybeta_{2p-1}^{(\mys)}  \mybeta_{2q}^{(\mys)} \frac{  (4n+1)(\mys-2q)(1+2p+2n)(\mys-p-n)  }
		{ (\mys-p)(1+2p)(\mys-q)(1+2q)} \varphi_{n} (\mys;p,q). 
	\end{align*}	
	Applying Lemma \ref{lem:alpha-beta} and using $(x)_{n+1} = (x)_{n} (x+n)$ 
	and $(x)_{n+1} = (x+1)_{n} x$, we can deduce  
	\begin{align*}
		 \myalpha_{2i,2p}^{(\mys)}  \myalpha_{2i,2q}^{(\mys)}
		 & \stackrel{\eqref{eq:AB1}}{=} 
		4(4i-1)\mybeta_{2p}^{(\mys)} \mybeta_{2q}^{(\mys)} 
		\frac{ \left( \mys+\frac12-p \right)_i (-p)_i }{ (p-\mys)_i \left( p+\frac12 \right)_i } 
		\frac{ \left( \mys+\frac12-q \right)_i (-q)_i }{ (q-\mys)_i \left( q+\frac12 \right)_i } 
		\\
		& = 4(4n+3)\mybeta_{2p-1}^{(\mys)}   \frac{  \mybeta_{2q}^{(\mys)} (\mys-2p+1) }{ 2p(2\mys -2p+1) } 	\frac{ \left( \mys+\frac12-p \right)_{n+1} (-p)_{n+1} }{ (p-\mys)_{n+1} \left( p+\frac12 \right)_{n+1} } 
		\frac{ \left( \mys+\frac12-q \right)_{n+1} (-q)_{n+1} }{ (q-\mys)_{n+1} \left( q+\frac12 \right)_{n+1} } 
		\\
		&= 2 \mybeta_{2p-1}^{(\mys)}  \mybeta_{2q}^{(\mys)} 
		\frac{ (4n+3)(1+\mys-2p)(q- n )(1+2\mys+2n-2q)  }
		{ (\mys-p)(1+2p)(\mys-q)(1+2q) }  \varphi_n (\mys;p,q).
	\end{align*}
	Combining the above two equations gives \eqref{eq:alpha_beta_sigma} and completes the proof. 
\end{proof}

\begin{lemma}
	For $p,q\in \mathbb Z^+$, define a sequence of rational functions of $\mys$: for $n=0,1,\dots$, 
	\begin{equation}\label{Def:Phi}
		\Phi_n(\mys;p,q) := 
		\frac{ {\mathcal C}_{n,p,q}^{3,\mys} }{ (\mys-p)(1+2p)(\mys-q)(1+2q) } \varphi_n(\mys;p,q)
	\end{equation}
	with ${\mathcal C}_{n,p,q}^{3,\mys} := (n+p-\mys)(1+2p+2n)(n+q-\mys)(1+2q+2n) $. 
	Then, for any $\mys \in \widehat{\mathbb R}$ and $p,q \in \mathbb Z^+$, we have 
	\begin{align}
		&\Phi_0(\mys;p,q) = 1, \label{Phi0}
		\\
		&\Phi_n(\mys;p,q) = 0 \qquad \forall n \ge p, \label{PhiN}
		\\
		& \Phi_{n+1}(\mys;p,q) - \Phi_{n}(\mys;p,q) = -\phi_n (\mys;p,q) \qquad 
		\forall n \in \mathbb N. \label{PhiN+1}
	\end{align}
\end{lemma}

\begin{proof}
	{\tt Proof of \eqref{Phi0}.}	Because $(x)_0=1$, we have $\varphi_0(\mys;p,q)=1$. Then by
	${\mathcal C}_{0,p,q}^{3,\mys} = (p-\mys)(1+2p)(q-\mys)(1+2q)$, we obtain 
	$\Phi_0(\mys;p,q) = \varphi_0(\mys;p,q) = 1$. 
	
	{\tt Proof of \eqref{PhiN}.}  Recall \eqref{eq:ngep} shows  $\varphi_n(\mys;p,q) = 0$ for all $n\ge p$. 
	This immediately leads to \eqref{PhiN}. 
	
	{\tt Proof of \eqref{PhiN+1}.} 	Utilizing the relation $(x)_{n+1} = (x)_n (x+n)$ gives 
	\begin{align*}
		\varphi_{n+1}(\mys;p,q) = \frac{  \left( \mys + \frac32  - p + n  \right) (1-p+n)  \left(  \mys+\frac12 - q +n \right) (n-q) }
		{  (p-\mys+1+n)  \left(   p + \frac32 + n \right) (q-\mys+1 +n  )
			\left(  q + \frac32 +n \right)   } \varphi_{n}(\mys;p,q)
		=: {\mathcal C}_{n,p,q}^{4,\mys} \varphi_{n}(\mys;p,q).
	\end{align*}
	It follows that 
	\begin{equation*}
		\Phi_{n+1}(\mys;p,q) = \frac{ {\mathcal C}_{n+1,p,q}^{3,\mys} \varphi_{n+1}(\mys;p,q) }{ (\mys-p)(1+2p)(\mys-q)(1+2q) }  
		= \frac{ {\mathcal C}_{n+1,p,q}^{3,\mys} {\mathcal C}_{n,p,q}^{4,\mys} \varphi_{n}(\mys;p,q)  }{ (\mys-p)(1+2p)(\mys-q)(1+2q) } 
	\end{equation*}
	with ${\mathcal C}_{n+1,p,q}^{3,\mys} {\mathcal C}_{n,p,q}^{4,\mys} = 
	(2n+2\mys-2p+3)(n-p+1)(2n+2\mys-2q+1)(n-q).$ 
	By direct calculations, we observe that the identity 
	$
	{\mathcal C}_{n+1,p,q}^{3,\mys} {\mathcal C}_{n,p,q}^{4,\mys} - {\mathcal C}_{n,p,q}^{3,\mys} = - {\mathcal C}_{n,p,q}^{1,\mys} - {\mathcal C}_{n,p,q}^{2,\mys},
	$ always holds, 
%	By direct calculations, we observe that the following identity always holds: 
%	$$
%	{\mathcal C}_{n+1,p,q}^{3,\mys} {\mathcal C}_{n,p,q}^{4,\mys} - {\mathcal C}_{n,p,q}^{3,\mys} = - {\mathcal C}_{n,p,q}^{1,\mys} - {\mathcal C}_{n,p,q}^{2,\mys},
%	$$
	which leads to  
	\begin{align*}
		\Phi_{n+1}(\mys;p,q) - \Phi_n(\mys;p,q) 
		&= \frac{   {\mathcal C}_{n+1,p,q}^{3,\mys} {\mathcal C}_{n,p,q}^{4,\mys} - {\mathcal C}_{n,p,q}^{3,\mys} }{ (\mys-p)(1+2p)(\mys-q)(1+2q) } \varphi_n(\mys;p,q) 
		\\
		&= \frac{  - {\mathcal C}_{n,p,q}^{1,\mys} - {\mathcal C}_{n,p,q}^{2,\mys} }{ (\mys-p)(1+2p)(\mys-q)(1+2q) } \varphi_n(\mys;p,q) 
		%\\	&
		= - \phi_n (\mys;p,q).
	\end{align*}

	%This completes the proof. 
\end{proof}

\begin{lemma}\label{lem:sumphi}
	For any $\mys\in \widehat{\mathbb R}$, the functions $\{\phi_n(\mys;p,q)\}$ defined in \cref{Def_phi} satisfy 
	\begin{equation}\label{sum_phin}
		\sum_{n=0}^{\infty} \phi_n(\mys;p,q) =  \sum_{n=0}^{p-1} \phi_n(\mys;p,q) =1\qquad  \forall p,q \in \mathbb Z^+.
	\end{equation}
\end{lemma}

\begin{proof}
	Recall that we have proven in \eqref{eq:ngep} that  $\phi_n(\mys;p,q) = 0$ for all $n\ge p$. Thus  
	the series \eqref{sum_phin} contains only finite sums. 
	This fact, together with \eqref{Phi0}--\eqref{PhiN+1}, implies that 
	\begin{equation*}
		\sum_{n=0}^{\infty} \phi_n(\mys;p,q)  =  \sum_{n=0}^{p-1} \phi_n(\mys;p,q)  
		%= - \sum_{n=0}^{p-1}  ( \Phi_{n+1}(m;p,q) - \Phi_{n}(m;p,q) ) 
		=  - \Phi_{p}(\mys;p,q) + \Phi_{0}(\mys;p,q) = - 0 + 1 =1.
	\end{equation*}

	%This completes the proof. 
\end{proof}

Combining the results in \cref{lem:alpha_beta_sigma,lem:sumphi}, we obtain the following crucial identity \eqref{sum_alpha-beta}. It is worth noting that the discovery of this identity \eqref{sum_alpha-beta} is highly nontrivial and become the key to proving \cref{prop:key-cholesky}.

\begin{lemma}\label{lem:sum_alpha-beta}
	For any $\mys\in \widehat{\mathbb R}$, we have 
	\begin{equation}\label{sum_alpha-beta}
		\sum_{i=1}^{\infty} \myalpha_{i,p}^{(\mys)} \myalpha_{i,q+1}^{(\mys)} 
		+ \sum_{i=1}^{\infty} \myalpha_{i,p+1}^{(\mys)} \myalpha_{i,q}^{(\mys)} = 2 \mybeta_p^{(\mys)} \mybeta_q^{(\mys)} \qquad 
		\forall p,q \in \mathbb Z^+,~ p \equiv q+1\pmod{2}. 
	\end{equation}
	Note the series in \eqref{sum_alpha-beta} is actually {\bf finite} sums, since 
	$\myalpha_{i,j}^{(\mys)}=0$ when $i>j$ by definition \eqref{eq:Defalpha}.  
\end{lemma}

\begin{proof}
	Observing that $p$ and $q$ are symmetric in \eqref{sum_alpha-beta} and $p \equiv q+1\pmod{2}$,    
	we assume, without loss of generality, that $p$ is odd and $q$ is even (otherwise, we can simply exchange $p$ and $q$), 
	and denote  
	$$p = 2 \widehat p - 1, \quad q = 2 \widehat q \qquad  \mbox{with} ~~\widehat p, \widehat q \in \mathbb Z^+.$$ 
	According to definition \eqref{eq:Defalpha},
	$\myalpha_{i,p}^{(\mys)}=0$ if $i$ is even, and $\myalpha_{i,q}^{(\mys)}=0$ if $i$ is odd. 
	Thus 
	\begin{equation} \label{eq:12882}
		\sum_{i=1}^{\infty} \myalpha_{i,p}^{(\mys)} \myalpha_{i,q+1}^{(\mys)} 
		+ \sum_{i=1}^{\infty} \myalpha_{i,p+1}^{(\mys)} \myalpha_{i,q}^{(\mys)} 
		= 	\sum_{i=1}^{\infty} \myalpha_{2i-1, 2 \widehat p - 1 }^{(\mys)} \myalpha_{2i-1,2 \widehat q+1}^{(\mys)} 
		+ \sum_{i=1}^{\infty} \myalpha_{2i,2 \widehat p}^{(\mys)} \myalpha_{2i,2 \widehat q}^{(\mys)}. 
	\end{equation}
	It follows from \cref{lem:alpha_beta_sigma,lem:sumphi} that 
	\begin{align*}
		\sum_{i=1}^{\infty} \myalpha_{2i-1, 2 \widehat p - 1 }^{(\mys)} \myalpha_{2i-1,2 \widehat q+1}^{(\mys)} 
		+ \sum_{i=1}^{\infty} \myalpha_{2i,2 \widehat p}^{(\mys)} \myalpha_{2i,2 \widehat q}^{(\mys)} & \stackrel{\eqref{eq:alpha_beta_sigma}}{=} 
		2 \sum_{i=1}^{\infty} \mybeta_{2 \widehat p-1}^{(\mys)} \mybeta_{2 \widehat q}^{(\mys)} \phi_{i-1}(\mys; \widehat p, \widehat q) 
		 \stackrel{\eqref{sum_phin}}{=} 
		2 \mybeta_{2 \widehat p-1}^{(\mys)} \mybeta_{2 \widehat q}^{(\mys)} 
		= 2 \mybeta_{p}^{(\mys)} \mybeta_{q}^{(\mys)}, 
	\end{align*}
	which along with \eqref{eq:12882} yields \eqref{sum_alpha-beta}.  
	The proof is completed. 
\end{proof}

\subsection{Proof of \cref{prop:key-cholesky}}\label{sec:cholesky-proof}

\begin{proof}
	Note that $\gamma_{p,q}=\gamma_{q,p}$, so that $p$ and $q$ are symmetric in \eqref{eq:key-cholesky}.  
	Without loss of generality, we assume in the following proof that $p \le q$. 
	The proof is divided into three parts. 
	
	{\tt (\romannumeral1) Prove \eqref{eq:key-cholesky} for $p\not\equiv q \pmod{2}$.} In this case, $(-1)^{p-1} + (-1)^{q-1} =0$, and thus $\gamma_{p-1,q-1}^{(s)}=0$.  By \eqref{eq:Defalpha}, we know for any given $i \in \mathbb Z^+$ that either $\myalpha_{i,p}^{(\mys)}=0$ or $\myalpha_{i,q}^{(\mys)}=0$. Therefore,  $\sum_{i=1}^{\min\{p,q\}} \myalpha_{i,p}^{(\mys)} \myalpha_{i,q}^{(\mys)} =0 = -\gamma_{p-1,q-1}$.

	{\tt (\romannumeral2) Prove \eqref{eq:key-cholesky} for the special case $q\ge p=1$ and $p\equiv q \pmod{2}$}, namely,
	\begin{equation}\label{case:p=1}
		\sum_{i=1}^{\min\{1,q\}} \myalpha_{i,1}^{(\mys)} \myalpha_{i,q}^{(\mys)} = -\gamma_{0,q-1}^{(\mys)} \qquad \forall q\ge 1,~~p\equiv q \pmod{2},
	\end{equation}
	where the left-hand-side term is  $\myalpha_{1,1}^{(\mys)} \myalpha_{1,q}^{(\mys)}$, and the right-hand-side term is 
	$-\gamma_{0,q-1}^{(\mys)}= 2 \mybeta_0^{(\mys)} \mybeta_q^{(\mys)}$ by \eqref{eq:Defg}. 
	Using \eqref{eq:AB2} and noting $q$ is odd in this case, we have $\myalpha_{1,1}^{(\mys)} \myalpha_{1,q}^{(\mys)} = 4 \theta_1^{(\mys)}\theta_q^{(s)} = 2 \mybeta_0^{(\mys)} \mybeta_q^{(\mys)}$. Hence \eqref{case:p=1} holds. 
	%This finishes the proof of \eqref{case:p=1}. 

	{\tt (\romannumeral3) Prove \eqref{eq:key-cholesky}  for $q \ge p > 1$ and $p\equiv q \pmod{2}$.}  
	Since $\myalpha_{i,j} = 0$ when $i>j$, %by definition \eqref{eq:Defalpha}, 
	we can rewrite 
	\begin{equation}\label{eq:WKL3903}
		\sum_{i=1}^{\min\{p,q\}} \myalpha_{i,p}^{(\mys)} \myalpha_{i,q}^{(\mys)} = \sum_{i=1}^{\infty} \myalpha_{i,p}^{(\mys)} \myalpha_{i,q}^{(\mys)}.
	\end{equation}
	We first give the following technical splittings (note all the series below are actually finite sums): 
	%Noting all the series are actually finite sums, we derive the technical splittings: 
	\begin{align*}
		\sum_{i=1}^{\infty} \myalpha_{i,p}^{(\mys)} \myalpha_{i,q}^{(\mys)}  
		& =  \sum_{k=0}^{p-1} (-1)^{k} 
		\sum_{i=1}^{\infty} \myalpha_{i,p-k}^{(\mys)} \myalpha_{i,q+k}^{(\mys)} 
		- \sum_{k=1}^{p-1} (-1)^{k} 
		\sum_{i=1}^{\infty} \myalpha_{i,p-k}^{(\mys)} \myalpha_{i,q+k}^{(\mys)} 
		\\
		%	& = \sum_{i=1}^{\infty} \myalpha_{i,p}^{(m)} \myalpha_{i,q}^{(m)}   
		%	+ \sum_{k=1}^{p-1} (-1)^{k} 
		% \sum_{i=1}^{\infty} \myalpha_{i,p-k}^{(m)} \myalpha_{i,q+k}^{(m)} 
		%	 + \sum_{k=1}^{p-1} (-1)^{k-1} 
		%	 \sum_{i=1}^{\infty} \myalpha_{i,p-k}^{(m)} \myalpha_{i,q+k}^{(m)} 
		%	 \\
		%		& = \sum_{i=1}^{\infty} \myalpha_{i,p}^{(m)} \myalpha_{i,q}^{(m)}   
		%	+ \sum_{k=2}^{p} (-1)^{k-1} 
		%	\sum_{i=1}^{\infty} \myalpha_{i,p-k+1}^{(m)} \myalpha_{i,q+k-1}^{(m)} 
		%	+ \sum_{k=1}^{p-1} (-1)^{k-1} 
		%	\sum_{i=1}^{\infty} \myalpha_{i,p-k}^{(m)} \myalpha_{i,q+k}^{(m)}  
		%	\\
		& =  \sum_{k=1}^{p} (-1)^{k-1} 
		\sum_{i=1}^{\infty} \myalpha_{i,p-k+1}^{(\mys)} \myalpha_{i,q+k-1}^{(\mys)} 
		+ \sum_{k=1}^{p-1} (-1)^{k-1} 
		\sum_{i=1}^{\infty} \myalpha_{i,p-k}^{(\mys)} \myalpha_{i,q+k}^{(\mys)}  	
		\\
		& = \sum_{k=1}^{p-1} (-1)^{k-1} \left( \sum_{i=1}^{\infty} \myalpha_{i,p-k+1}^{(\mys)} \myalpha_{i,q+k-1}^{(\mys)}  + 
		\sum_{i=1}^{\infty} \myalpha_{i,p-k}^{(\mys)} \myalpha_{i,q+k}^{(\mys)} \right)  +  (-1)^{p-1} 
		\sum_{i=1}^{\infty} \myalpha_{i,1}^{(\mys)} \myalpha_{i,q+p-1}^{(\mys)}.
	\end{align*}
	Applying \cref{lem:sum_alpha-beta} with $\tilde p = p-k \in \mathbb Z^+$, $\tilde q = q+k-1 \in \mathbb Z^+$, and $\tilde p \equiv \tilde q+1\pmod{2}$, 
	we get 
	$$
	\sum_{i=1}^{\infty} \myalpha_{i,p-k}^{(\mys)} \myalpha_{i,q+k}^{(\mys)} 
	+ \sum_{i=1}^{\infty} \myalpha_{i,p-k+1}^{(\mys)} \myalpha_{i,q+k-1}^{(\mys)} = 2 \mybeta_{p-k}^{(\mys)} \mybeta_{q+k-1}^{(\mys)}.
	$$
	Therefore, %we obtain 
	\begin{align*}
		\sum_{i=1}^{\infty} \myalpha_{i,p}^{(\mys)} \myalpha_{i,q}^{(\mys)}  
		& = \sum_{k=1}^{p-1} (-1)^{k-1} \left( 2 \mybeta_{p-k}^{(\mys)} \mybeta_{q+k-1}^{(\mys)} \right) 
		+  (-1)^{p-1} 
		\sum_{i=1}^{\infty} \myalpha_{i,1}^{(\mys)} \myalpha_{i,q+p-1}^{(\mys)}
		\\
		& \stackrel{\eqref{case:p=1}}{=}  \sum_{k=1}^{p-1} (-1)^{k-1} \left( 2 \mybeta_{p-k}^{(\mys)} \mybeta_{q+k-1}^{(\mys)} \right) 
		+  (-1)^{p-1} \left( 2 \mybeta_0^{(\mys)} \mybeta_{q+p-1}^{(\mys)} \right)
		\\
		& = 2 \sum_{k=1}^{p} (-1)^{k-1}  \mybeta_{p-k}^{(\mys)} \mybeta_{q+k-1}^{(\mys)} 
		= 2 \sum_{j=0}^{p-1} (-1)^{p-j-1}  \mybeta_{j}^{(\mys)} \mybeta_{p+q-j-1}^{(\mys)} 
		%\qquad \qquad \quad (j:=p-k)
		\\
		&= 2 (-1)^{p-1} \sum_{j=0}^{p-1} (-1)^{j}  \mybeta_{j}^{(\mys)} \mybeta_{p+q-j-1}^{(\mys)} = -\gamma_{p-1,q-1}^{(\mys)}.
	\end{align*}
	This together with \cref{eq:WKL3903} completes the proof of \cref{prop:key-cholesky}. 
\end{proof}
%
%For {$n=3$}, we have 
%\begin{align*}
%	{\bf G}_3^{(\mys)} &= \begin{pmatrix}
%		1 & 0 & \frac{\mys-2}{ 12 (2\mys-1) } & 0 \\
%		0 & \frac1{12} & 0 & \frac{\mys-3}{ 240 (2\mys-1)  } \\
%		\frac{\mys-2}{ 12 (2\mys-1) } & 0 & \frac{3 \mys^2 - 8 \mys + 7}{240(2\mys - 1)^2} & 0 \\
%		0 & \frac{\mys-3}{ 240 (2\mys-1)  } & 0 & \frac{5 \mys^2 - 26 \mys + 38}{20160(2\mys - 1)^2} 
%	\end{pmatrix},
%\\
%	{\bf U}_3^{(\mys)} &= \begin{pmatrix}
%		1 & 0  & \frac{\mys-2}{ 12 (2\mys-1) } & 0 \\
%		0 & \frac{ \sqrt{3} }6 & 0 & \frac{\sqrt{3}(\mys - 3)}{120(2\mys - 1)} \\
%		0 & 0 & \frac{ \sqrt{5} }{60}  & 0 \\
%		0 & 0 & 0 & \frac{ \sqrt{7} }{840}
%	\end{pmatrix}.
%\end{align*}

\section{Numerical results}\label{sec:num}
This section gives a few numerical examples to confirm the theoretical results. 

\begin{example}\label{ex1}
	The first example considers a linear seminegative system from \cite{sun2017rk4}: 
	$$
	\frac{d}{dt} u = L u, \quad u = u(t) \in L^2( [0,T], \mathbb R^3 ), \quad L = - \begin{pmatrix}
		1 & 2 & 2 \\
		0 & 1 & 2 \\
		0 & 0 & 1
	\end{pmatrix}.
	$$
	The $(s,s)$ diagonal Pad\'e approximations with $s=3$ and $s=4$ are used to solve this system 
	with an arbitrarily chosen initial condition $u(0)=(0.9134,0.2785,0.5469)^\top$ up to $t=8$. 
	In order to verify the convergence, we run the simulations with  
	different time stepsizes $\tau \in \{1.6, 0.8, 0.4, 0.2\}$.  
	The $l^2$-errors the numerical solutions and the energy dissipation accuracy (see \cref{rem:DE_accuracy} for the definition) are listed in 
	\cref{tab:acc2}. 
	We observe the convergence rate of $2s$ for the $(s,s)$ diagonal Pad\'e approximation, as expected. 
	We also plot the energy dissipation magnitudes $\nm{u^n}^2 - \nm{u^{n+1}}^2$ over time in \cref{fig:EnergyEx1}. 
	One can observe that $\nm{u^n}^2 - \nm{u^{n+1}}^2$ is always positive, which indicates the energy decay property as expected from the unconditionally strong stability in \cref{thm:pade-ed-series}. 
	Moreover, the numerical energy dissipation magnitudes agree well with the theoretical ones, 
	which further confirms the correctness of our energy identity \cref{eq:pade-ed-series}. 
\end{example}

\begin{table}[htbp]\label{tab:acc2}
	\centering
	\captionsetup{belowskip=-10pt}
	\caption{\small The 
		$l^2$-errors and energy dissipation accuracy $\Delta E$ at $t=8$, and the corresponding convergence rates for
		the $(s,s)$ diagonal Pad\'e approximations.
	}
	\label{tab:order2}
	\begin{tabular}{c|c|c|c|c|c|c|c|c}
		\hline
		\multirow{2}{8pt}{$\tau$}
		&\multicolumn{4}{c|}{$s=3$}
		&\multicolumn{4}{c}{$s=4$} 
		\\
		\cline{2-9}
		& $l^2$ error& order & $\Delta E$ & order & $l^2$ error& order & $\Delta E$ & order \\
		\hline
		$1.6$& 3.56e-6 & --          &1.35e-7& --   &  2.77e-8 & -- & 1.07e-9  &-- \\
		$0.8$& 5.25e-8 & 6.09   & 1.98e-9&    6.09 & 1.12e-10 &  7.96 & 4.34e-12 & 7.95 \\
		$0.4$& 8.07e-10 & 6.02   & 3.05e-11&  6.02 & 4.39e-13 & 7.99  & 1.71e-14 & 7.99 \\
		$0.2$& 1.26e-11 & 6.01   & 4.74e-13&  6.01 & 1.64e-15 & 8.07 & 6.36e-17 & 8.07\\
		\hline
	\end{tabular}
\end{table}

\begin{example}\label{ex2}
	This example investigates the following seminegative ODE system
\begin{equation}\label{keyex2}
		\frac{d}{dt} u = L u, \quad u = u(t) \in L^2( [0,T], \mathbb R^{2 N_d} ), \quad L =  \frac1{\Delta x} \begin{pmatrix}
		{\bf L}_1 & \sqrt{3} {\bf L}_1 
		\\
		\sqrt{3} ( 2 {\bf I}_{N_d} -  {\bf L}_2) & -3 {\bf L}_2
	\end{pmatrix}
\end{equation}	
	with 
\begin{equation}\label{eq:L1L2}
		{\bf L}_1 :=
	\begin{pmatrix}
		-1 &  &  & 1 \\
		1 & \ddots &  &   \\
		 & \ddots & \ddots &  \\
		 &  & 1 & -1 
	\end{pmatrix}, \qquad 
	{\bf L}_2 :=
\begin{pmatrix}
	1 &  &  & 1 \\
	1 & \ddots &  &   \\
	& \ddots & \ddots &  \\
	&  & 1 & 1 
\end{pmatrix}.
\end{equation}
	This system arises from the piecewise linear (${\mathbb P}^1$-based) discontinuous Galerkin discretization \cite{CockburnShu1998} of the linear convection PDE $\psi_t + \psi_x=0$ 
	in the spatial domain $[0,1]$ with the uniform mesh of $N_d=20$ cells (i.e., $\Delta x=1/N_d=0.05$) and periodic boundary conditions. The initial solution is taken as $\psi(x,0)=\sin(2\pi x)$. 
	We solve the semi-discrete ODE system \cref{keyex2} in time up to $t=4$ by using the $(2,2)$ 
	diagonal Pad\'e approximation. Due to its unconditional strong stability (\cref{thm:pade-ed-series}), 
	a large time stepsize $\tau = 0.1$ is used and works robustly. 
	The energy dissipation information shown in \cref{fig:EnergyEx2} further validates our theoretical 
	energy laws \cref{eq:pade-ed-series} and stability analysis.
\end{example}

%	\begin{figure}[htbp]
%	\centering
%	{\includegraphics[width=0.32\textwidth]{Exp1_dt0_2_PadeM3_T8.eps}}
%	{\includegraphics[width=0.32\textwidth]{Exp3_dt0_1_PadeM2_T4.eps}}
%	{\includegraphics[width=0.32\textwidth]{Exp2_dt0_1_PadeM2_T4.eps}}
%	\caption{\small Example \ref{ex:rst}: 
%		The numerical energy dissipation magnitudes and the theoretical ones given by the right-hand-side term of the energy equality \cref{eq:pade-ed-series} for \cref{ex1,ex2,ex3}}   		
%	\label{fig:rst}
%\end{figure}

\begin{figure}[htbp]
	\centering
	\begin{subfigure}[b]{0.3\textwidth}
		\centering
		\includegraphics[width=\textwidth]{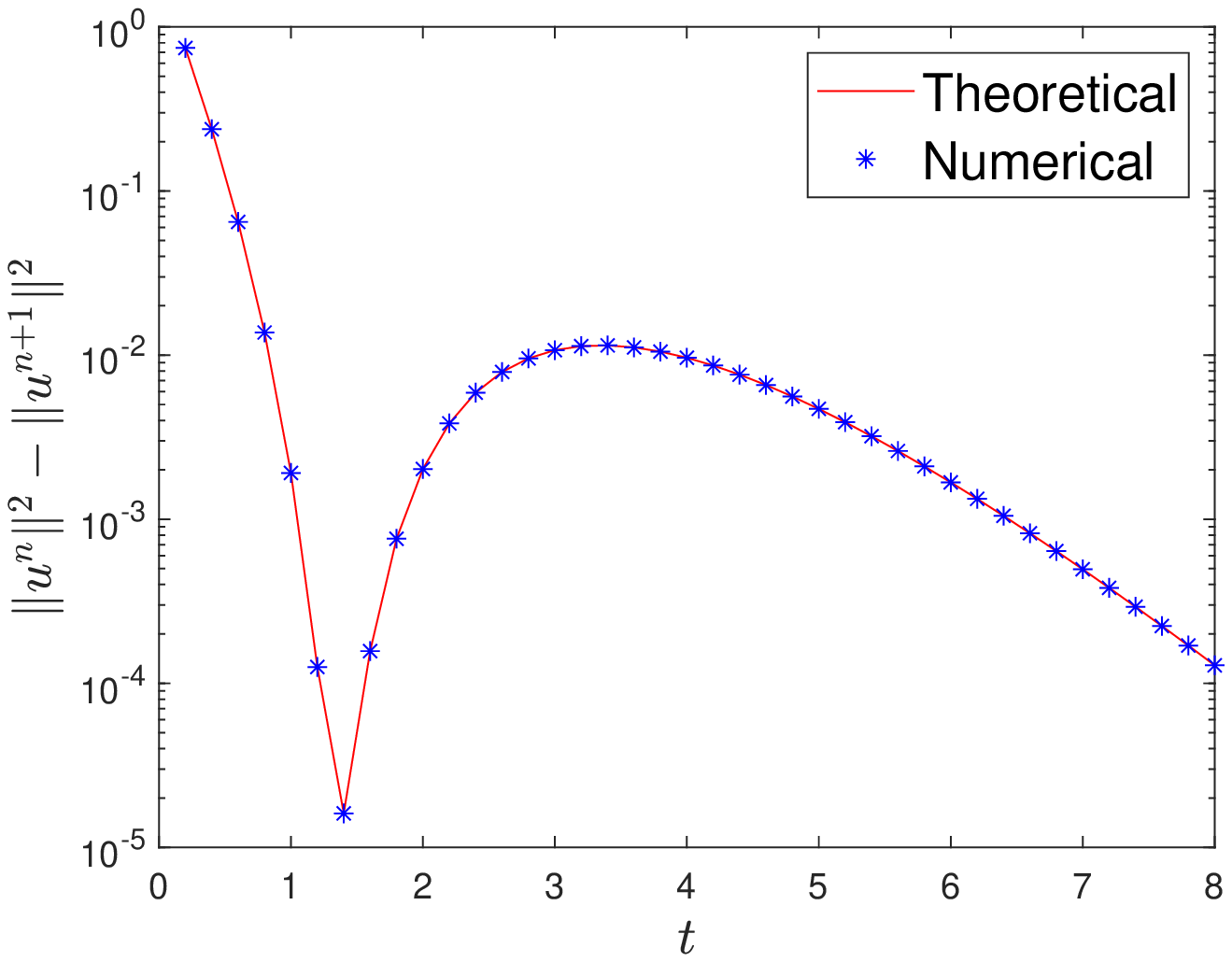}
		\captionsetup{belowskip=-7pt}
		\caption{\cref{ex1}}
		\label{fig:EnergyEx1}
	\end{subfigure}
	\hfill
	\begin{subfigure}[b]{0.3\textwidth}
		\centering
		\includegraphics[width=\textwidth]{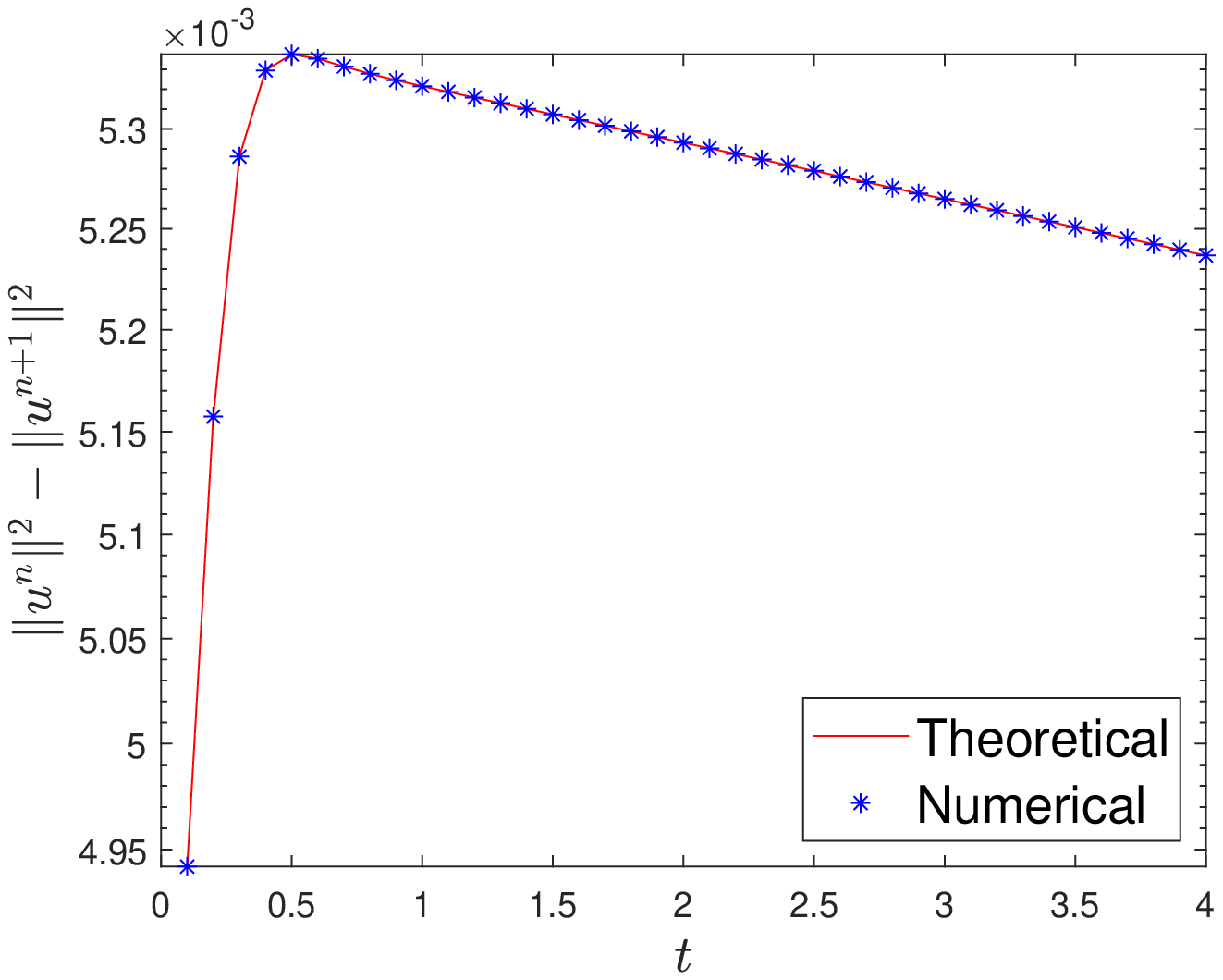}
		\captionsetup{belowskip=-7pt}
		\caption{\cref{ex2}}
		\label{fig:EnergyEx2}
	\end{subfigure}
	\hfill
	\begin{subfigure}[b]{0.3\textwidth}
		\centering
		\includegraphics[width=\textwidth]{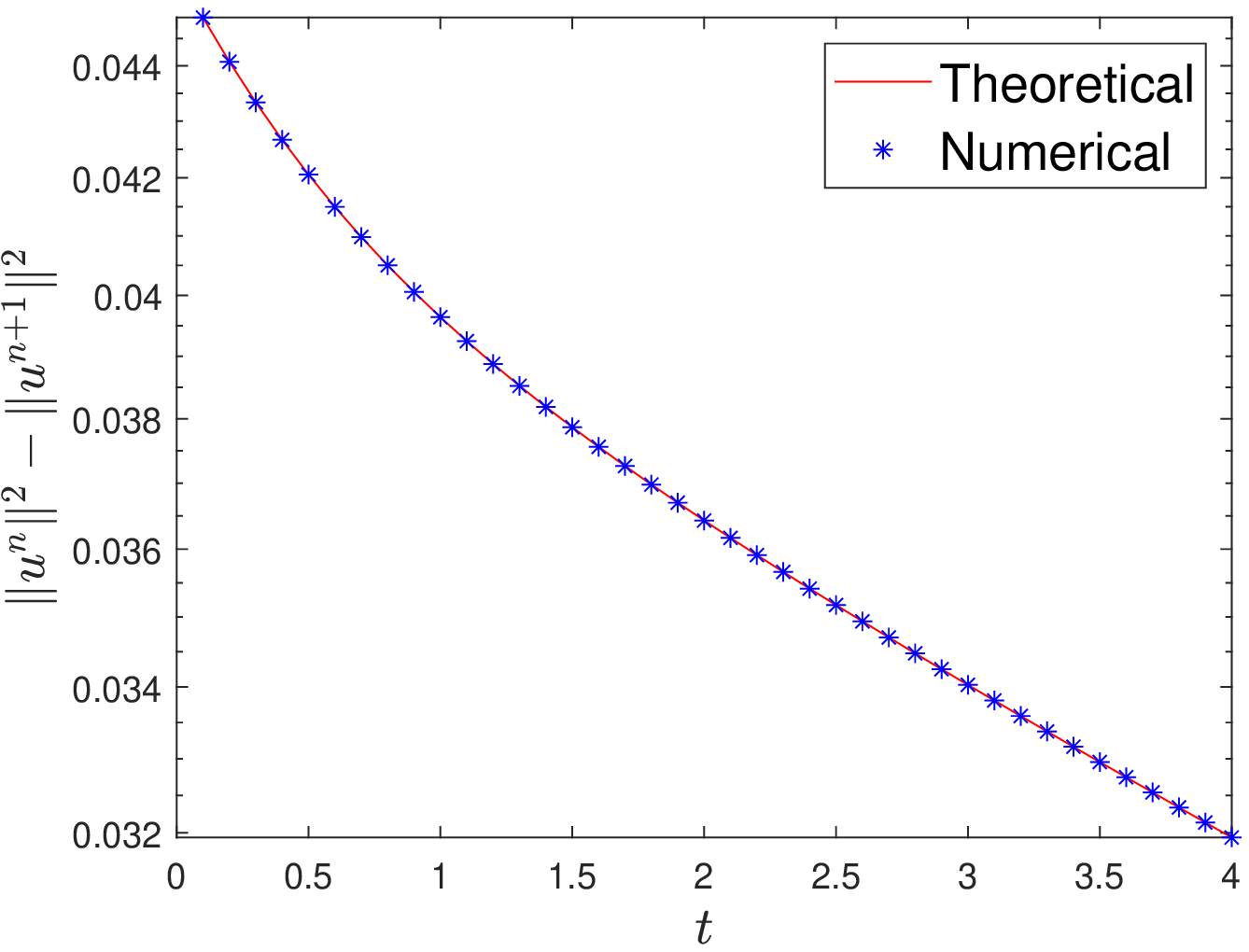}
		\captionsetup{belowskip=-7pt}
		\caption{\cref{ex3}}
		\label{fig:EnergyEx3}
	\end{subfigure}
	\captionsetup{belowskip=-6pt}
	\caption{\small Numerical energy dissipation magnitudes and the theoretical ones given by the energy identity \cref{eq:pade-ed-series}.}
	\label{fig:Energy}
\end{figure}

\begin{example}\label{ex3}
	In this example, we study the following seminegative ODE system
\begin{equation}\label{keyex3}
		\frac{d}{dt} u = L u, \quad u = u(t) \in L^2( [0,T], \mathbb R^{N_d} ), \quad L =  \frac1{\Delta x^3} {\bf L}_1 {\bf L}_1^\top {\bf L}_1^\top 
\end{equation}
	with the matrix ${\bf L}_1$ defined by \cref{eq:L1L2}. 
	This system comes from the piecewise constant (${\mathbb P}^0$-based) local discontinuous Galerkin discretization of the dispersion PDE $\psi_t + \psi_{xxx}=0$ 
	in the spatial domain $[0,1]$ with the uniform mesh of $N_d=20$ cells (i.e., $\Delta x=1/N_d=0.05$) and periodic boundary conditions. 
	The initial solution is taken as $\psi(x,0)=\cos(2\pi x)$. 
	We solve the semi-discrete ODE system \cref{keyex3} in time up to $t=4$ by using the $(2,2)$ 
	diagonal Pad\'e approximation. 
	The unconditional stability proved in \cref{thm:pade-ed-series} allows us to 
	use a much larger time stepsize $\tau = 0.1$, which is not restricted 
	by the normal CFL condition $\Delta t \le C \Delta x^3$ for an explicit time discretization such system \cref{keyex3}. 
	\cref{fig:EnergyEx3} displays the energy dissipation behavior, which is consistent with our theoretical  analysis.
\end{example}

\section{Conclusions}\label{sec:conclusions}
We have established a systematic theoretical framework to 
derive the discrete energy laws 
of {\em general implicit and explicit} Runge--Kutta (RK) methods for linear seminegative systems.  
The framework is motivated by a discrete analogue of integration by parts technique 
and a series expansion of the continuous energy law. 
The established discrete energy laws show a precise characterization
on whether and how the energy dissipates in the RK discretization, 
thereby giving stability criteria of RK methods. 
We have also found a unified discrete energy law for all the diagonal Pad\'e approximations,  
based on analytically constructing the Cholesky type decomposition of 
a class of symmetric matrices, whose structure is highly complicated. 
The discovery of the unified energy law and the proof of the decomposition
are very nontrivial.  
For the diagonal Pad\'e approximations, our analyses have bridged   
the continuous and discrete energy laws, enhancing our understanding of their intrinsic mechanisms. 
We have provided 
several specific examples of implicit methods to illustrate the discrete energy laws. 
 A few numerical examples have also been given to confirm the theoretical properties. 
In this paper, we have developed new analysis techniques, with 
construction of technical 
combinatorial identities and the theory of hypergeometric series, which were rarely used in previous RK stability analyses and  may motivate   
future developments in this field.

\appendix

\section{Proof of  \cref{lem:quadraticform}} \label{sec:proof:quadraticform}

\begin{proof}
	For any $v\in V$, we have
	\begin{align*}
	& \sum_{i = 0}^\myN	\sum_{j = 0}^\myN \gamma_{i,j}\tau^{i+j+1} \qip{L^i v, L^j v}
		= - 	\sum_{i = 0}^\myN	\sum_{j = 0}^\myN  \left(\sum_{k = 0}^{\myN}\mu_{k,i} \mylambda_k \mu_{k,j}\right) \tau^{i+j+1} \qip{L^i v, L^j v}\\ 
		=& -\sum_{k = 0}^\myN \mylambda_k \tau \left( 	\sum_{i = 0}^\myN	\sum_{j = 0}^\myN  \mu_{k,i}  \mu_{k,j} \tau^{i+j} \qip{L^i v, L^j v }\right)
		= -\sum_{k = 0}^\myN \mylambda_k \tau \qip{\sum_{i = k}^\myN \tau^i \mu_{k,i} L^i v,  \sum_{j=k}^\myN\mu_{k,j}\tau^jL^j v }\\
		=& -\sum_{k = 0}^\myN \mylambda_k \tau \qnm{\sum_{i = k}^\myN \tau^j \mu_{k,j} L^j v}^2
		= -\sum_{k = 0}^\myN \mylambda_k \tau^{2k+1} \qnm{L^k\left(\sum_{j = k}^\myN \mu_{k,j}(\tau L)^{j-k}\right)  v}^2.
	\end{align*}
	
\end{proof}

\section{Proof of \cref{lem:cholHilbert}}\label{sec:proof:cholHilbert}

\begin{proof}
	Observe that
	$
		\pmb{\widehat{\Upsilon}} = - \bfmyLambda_0 \mathbf{H} \bfmyLambda_0, 
	$
	where $\bfmyLambda_0 = \diag(\{\mylambda_i\}_{i=0}^\myN)$ with $\mylambda_i := 1/i!$, and $\mathbf{H} = (h_{i,j})_{i,j=0}^\myN$ is the Hilbert matrix with $h_{i,j} := {1}/{(i+j+1)}$. The Cholesky decomposition of the Hilbert matrix $\mathbf{H}$ gives 
	$
		\mathbf{{H}} = \mathbf{{U}}_H^\top \bfmyLambda_H \mathbf{{U}}_H,
	$ 
	where the formulae of $\mathbf{U}_H$ and $\bfmyLambda_H$ were given in \cite[Section 2]{hitotumatu1988cholesky} and also studied in \cite[Lemma 2]{hairer1981algebraically}. Therefore, we have
	\begin{equation*}
		\pmb{\widehat{\Upsilon}} = - \bfmyLambda_0 \mathbf{U}_H^\top \bfmyLambda_H \mathbf{U}_H \bfmyLambda_0 = - \left(\bfmyLambda_0^{-1} \mathbf{U}_H\bfmyLambda_0\right)^\top \left(\bfmyLambda_0 \bfmyLambda_H\bfmyLambda_0\right)\left(\bfmyLambda_0^{-1} \mathbf{U}_H \bfmyLambda_0\right). 
	\end{equation*}
	Taking $\mathbf{\widehat{U}} = \bfmyLambda_0^{-1} \mathbf{U}_H \bfmyLambda_0$ and $\bfhatmyLambda = \bfmyLambda_0 \bfmyLambda_H\bfmyLambda_0$ with the formulae of $\mathbf{U}_H$ and $\bfmyLambda_H$ from \cite[Section 2]{hitotumatu1988cholesky}, we obtain \cref{eq:lambdamu_hat} and complete the proof.
\end{proof}

\section{Proof of \cref{thm:ode-ed-series}}\label{sec:proof:ode-ed-series}

%The rest of the section is on the .
\begin{proof}
%We adopt the Cauchy--Kovalevskaya procedure to estimate the energy dissipation term. To be more specific, we consider the $\myN$th order Taylor series of $u(t^n+\tau)$ and then apply the differential equation \eqref{eq-odes}. It yields

Because $\sum_{i=0}^\infty \nm{ \frac{1}{i!}({\tau} L)^i u(t^n) } \le \sum_{i=0}^\infty  \frac{1}{i!}({\tau} \nm{ L} )^i \nm{ u(t^n) } \le  e^{\tau \nm{ L} } \nm{ u(t^n) } \le  e^{ T \nm{ L} } \nm{ u(t^n) } < \infty$, we known that the series 
$\sum_{i =0}^{\infty} \frac{1}{i!}({\tau} L)^i u(t^n)$ converges. 
This implies that $v(t^n+\tau) :=  \sum_{i =0}^{\infty} \frac{1}{i!}({\tau} L)^i u(t^n)$ is well-defined. We can verify that $\frac{d}{d \tau} v = L v $. By the uniqueness of the solution to \cref{eq-odes}, 
we get $u(t^n+\tau) = v(t^n+\tau) = \sum_{i =0}^{\infty} \frac{1}{i!}({\tau} L)^i u(t^n)$. 
 Define
$ 
u_\myN(t^n+{\tau}) %= \sum_{i =0}^{\myN} \frac{1}{i!}\left(\frac{d}{dt}\right)^i u(t^n) 
:= \sum_{i =0}^{\myN} \frac{1}{i!}({\tau} L)^i u(t^n). 
$ 
%For a fixed $\tau$ and any $\widehat{\tau} \in [0,\tau]$, 
%$$
%\nm{ u_N(t^n+{\hat \tau})-u(t^n+{\hat \tau})  } = \nm{  \sum_{i =N+1}^{\infty} \frac{1}{i!}({\hat \tau} L)^i u(t^n) } 
%\le \sum_{i =N+1}^{\infty}  \frac{1}{i!}({\tau} \nm{ L} )^i \nm{ u(t^n) } \to 0, \text{  as  } N \to \infty,
%$$
%which implies $\nm{ u_N(t^n+{\hat \tau})-u(t^n+{\hat \tau})  }$ converges to zero uniformly  for $\widehat{\tau} \in [0,\tau]$. It follows that $\qnm{u_N}$ converges to $\qnm{u}$ uniformly on $[0,\tau]$. Therefore, we have 
As $N \to \infty$, we have $\nm{u_N-u} \to 0$ and thus $\qnm{u_N}\to \qnm{u}$. 
It then follows from \cref{eq:ode-ed} that 
\begin{equation}\label{eq:ode-ex-2}
		\nm{u(t^{n+1})}^2 - \nm{u(t^n)}^2 = -\int_0^\tau \qnm{u(t^n+\widehat{\tau})}^2 \mathrm{d}\widehat{\tau}
		= -  \int_0^\tau \lim_{\myN\to \infty}\qnm{u_\myN(t^n+\widehat{\tau})}^2 \mathrm{d}\widehat{\tau}.
\end{equation}
Using the inequality $\qnm{u}^2 \le 2 \nm{L} \nm{u}^2$, we deduce that 
\begin{equation}\label{eq:WKL3112}
\qnm{u_\myN(t^n+\widehat{\tau})}^2 \le  2 \nm{L}  \nm{u_\myN(t^n+\widehat{\tau})}^2 \le 2\nm{L}\left(\sum_{i = 0}^N \frac{1}{i!}\widehat{\tau}^i\nm{ L}^i\right)^2 \nm{u(t^n)}^2	 \le 2 \nm{L} e^{2\widehat{\tau}\nm{L}}\nm{u(t^n)}^2.
\end{equation}
Thanks to the {\em dominated convergence theorem}, the estimate \cref{eq:WKL3112} along with 
$
	\int_0^\tau  2 \nm{L} e^{2\widehat{\tau}\nm{L}}\nm{u(t^n)}^2 \mathrm{d}\widehat{\tau}  = (e^{2\tau \nm{L}}-1)\nm{u(t^n)}^2<\infty  
$
implies 
$\int_0^\tau \lim_{\myN\to \infty}\qnm{u_\myN(t+\widehat{\tau})}^2 \mathrm{d}\widehat{\tau}
= \lim_{\myN\to \infty} \int_0^\tau \qnm{u_\myN(t+\widehat{\tau})}^2 \mathrm{d}\widehat{\tau}.$ 
Combining it with \cref{eq:ode-ex-2} gives 
\begin{equation}\label{eq:ode-ex-3}
	\nm{u(t^{n+1})}^2 - \nm{u(t^n)}^2 = -\lim_{\myN\to \infty} \int_0^\tau \qnm{u_\myN(t+\widehat{\tau})}^2 \mathrm{d}\widehat{\tau}.
\end{equation}
On the other hand, we can reformulate the integration \cref{eq:ode-ex-3} as follows 
\begin{align}\nonumber
		 \int_0^\tau \qnm{u_\myN(t^n+\widehat{\tau})}^2 d\widehat{\tau} 
		&= \int_0^\tau \qip{\sum_{i =0}^{\myN} \frac{1}{i!}(\widehat{\tau} L)^i u(t^n),  \sum_{j =0}^{\myN} \frac{1}{j!}(\widehat{\tau} L)^j u(t^n)} d\widehat{\tau} 
		\\ \nonumber
		&=  \sum_{i,j = 0}^{\myN} \left(\int_{0}^\tau\frac{\widehat{\tau}^{i+j}}{i!j!} \mathrm{d}\widehat{\tau}\right) \qip{L^i u(t^n), L^j u(t^n)}
		\\  
		&
		= \sum_{i,j = 0}^{\myN}\frac{\widehat{\tau}^{i+j+1}}{i!j!(i+j+1)} \qip{L^i u(t^n), L^j u(t^n)}
		=  \sum_{k = 0}^\myN \widehat{\mylambda}_k \tau^{2k+1} \qnm{L^k\widehat{u}_\myN^{(k)}}^2, 
		\label{eq-ode-ex}
\end{align} 
where the last equality follows from \cref{lem:quadraticform,lem:cholHilbert}, $\widehat{u}_\myN^{(k)} := \sum_{j = k}^\myN \widehat{\mu}_{k,j}(\tau L)^{j-k} u(t^n)$, and $\widehat{\mylambda}_k$ and $\widehat{\mu}_{k,j}$ are defined in \eqref{eq:lambdamu_hat}. 
Hence, by combining \cref{eq-ode-ex} with \cref{eq:ode-ex-3}, we obtain  
\begin{equation}\label{eq:WKL365}
		\nm{u(t^{n+1})}^2 - \nm{u(t^n)}^2= -\lim_{\myN\to \infty} \sum_{k = 0}^\myN \widehat{\mylambda}_k \tau^{2k+1} \qnm{L^k\widehat{u}_\myN^{(k)}}^2 = -\lim_{\myN\to \infty} \sum_{k = 0}^{\infty} \widehat{\mylambda}_k \tau^{2k+1} \qnm{L^k\widehat{u}_\myN^{(k)}}^2 1_{\{0,1,\dots, \myN\}}(k),
\end{equation}
where $1_{\{0,1,\dots N\}}(\cdot)$ is the indicator function. 
Note that 
\begin{equation}\label{WKL331}
	\widehat{\mylambda}_k \tau^{2k+1} \qnm{L^k\widehat{u}_\myN^{(k)}}^2 1_{\{0,1,\cdots, \myN\}}(k) 
	\leq  2 \tau \| L \| \| u(t^n) \|^2  \left( \sum_{j=k}^N \sqrt{\widehat{\mylambda}_k}\widehat{\mu}_{k,j}  (\tau\| L \| )^j  \right)^2  =: {\mathcal B}_k.
\end{equation}
The upper bound ${\mathcal B}_k$ satisfies 
\begin{equation}\label{WKL332}
\sum_{k=0}^\infty {\mathcal B}_k \le \| u(t^n) \|^2 \left(  e^{ 2 \tau \| L \|  } -1 \right) <\infty,
\end{equation}
because for any integer $M \ge N$,  
\begin{align*}
\sum_{k=0}^M {\mathcal B}_k & \le 2 \tau \| L \| \| u(t^n) \|^2 \sum_{k=0}^M \left( \sum_{j=k}^M \sqrt{\widehat{\mylambda}_k}\widehat{\mu}_{k,j}  (\tau\| L \| )^j  \right)^2
\\
&= 2 \tau \| L \| \| u(t^n) \|^2 \sum_{i=0}^M \sum_{j=0}^M \frac{ (\tau \| L \| )^{i+j} }{ i! j! (i+j+1) } 
= 2 \| u(t^n) \|^2 \int_{0}^{ \tau \| L \| } \left( \sum_{i=0}^M \frac{x^i}{i!}  \right)^2 {\rm d} x 
\\
& \le 2 \| u(t^n) \|^2 \int_{0}^{ \tau \| L \| } e^{ 2x } {\rm d} x  
= \| u(t^n) \|^2 \left(  e^{ 2 \tau \| L \|  } -1 \right),
\end{align*}
where we have used \cref{lem:cholHilbert} in the first equality. 
%which implies %$\sum_{k=0}^\infty a_k$ exists and 
%$$
%\sum_{k=0}^\infty a_k \le \| u(t^n) \|^2 \left(  e^{ 2 \tau \| L \|  } -1 \right) <\infty.
%$$ 
%
%Note that $0\leq \sqrt{\widehat{\mylambda}_k}\widehat{\mu}_{k,j} \leq 1/j!$. Together with Cauchy--Schwarz inequality, we have
%\begin{align*}
%	\widehat{\mylambda}_k \tau^{2k+1} \qnm{L^k\widehat{u}_\myN^{(k)}}^2 1_{\{0,1,\cdots, \myN\}}(k) 
%	\leq \sum_{j = k}^\infty \frac{1}{j!}(\tau \nm{L})^{j+1} \nm{u(t^n)}^2. 
%\end{align*}
%and 
%\begin{equation*}
%	\begin{aligned}
%		\sum_{k = 0}^\infty \left(\sum_{j = k}^\infty \frac{1}{j!}(\tau \nm{L})^{j+1} \nm{u(t^n)}^2\right)  =& \sum_{j = 0}^\infty \sum_{k = 0}^j \frac{1}{j!}(\tau \nm{L})^{j+1} \nm{u(t^n)}^2\\
%		=& \sum_{j = 0}^\infty \frac{(j+1)^2}{(j+1)!}(\tau \nm{L})^{j+1} \nm{u(t^n)}^2<\infty. 
%	\end{aligned}
%\end{equation*}
Due to \cref{WKL331,WKL332}, we  can again invoke the dominated convergence theorem to exchange the limit and the infinite summation in \cref{eq:WKL365} to obtain
\begin{equation*}
		\nm{u(t^{n+1})}^2 - \nm{u(t^n)}^2=- \sum_{k = 0}^{\infty} \lim_{\myN\to \infty} \widehat{\mylambda}_k \tau^{2k+1} \qnm{L^k\widehat{u}_\myN^{(k)}}^2 1_{\{0,1,\cdots, \myN\}}(k)
		 =- \sum_{k = 0}^{\infty} \widehat{\mylambda}_k \tau^{2k+1} \qnm{L^k\widehat{u}^{(k)}}^2,
\end{equation*}
which completes the proof. %\hfill \proofbox
\end{proof}

%\section{Proof of \cref{lem:eeconditional}}\label{app:eeconditional}

\section{Proof of \cref{lem:uuhat-cid}}\label{proof:uuhat-cid}

%\subsection{A combinatorial identity}

\begin{proof}
	When $i>j$ or $i=j$, the identity is obviously true. 
	In the following, 
	we only focus on the case of $i < j$. 
	Define 
	$$a_\kk := \binom{\mys-\kk}{j-\kk}^{-1}\binom{2\mys-\kk}{j-i-\kk}\binom{i+j+1}{\kk}(-1)^\kk
	= \frac{ (j-\kk)! (s-j)! (2s-\kk)! (i+j+1)!  }{ (s-\kk)!  (j-i-\kk)! (2s-j+i)! \kk! (i+j+1-\kk)! } (-1)^\kk.
	$$
	Then we have %\zsc{Where is $a_{l+1}/a_l$ used in the computation? Should we compute $a_{l}/a_0$ instead?}
	$$a_0 = \frac{ j! (s-j)! (2s)!  }{ s!  (j-i)! (2s-j+i)! }, \qquad %\frac{ a_{\kk+1} }{ a_\kk } = ,\quad 
	\frac{a_\kk}{a_0} = 
	\prod_{k=0}^{\kk-1} \frac{a_{k+1}}{a_{k}}
	= \prod_{k=0}^{\kk-1} \left( \frac{ ( k-j+i ) (k-s) ( k-i-j-1 )  }{ (k-2s)  (k-j)} \frac{1}{k+1} \right).$$
	Using the rising factorial notation \cref{eq:Def_Pochhammer},
	 one can reformulate the sum in \cref{lem:uuhat-cid} as
\begin{equation}\label{eq:WKL00}
		\sum_{\kk = 0}^{j-i} a_\kk = a_0 \sum_{\kk = 0}^{j-i} \frac{  (i-j)_\kk (-s)_\kk ( -i-j-1 )_\kk  }{ (-2s)_\kk (-j)_\kk   } \frac{1}{ \kk ! }
	= a_0 \sum_{\kk = 0}^{\infty} \frac{  (i-j)_\kk (-s)_\kk ( -i-j-1 )_\kk  }{ (-2s)_\kk (-j)_\kk   } \frac{1}{ \kk ! }, 
\end{equation}
	where we have used the fact $(i-j)_\kk = 0$ for $\kk > j-i$ and $\kk \in \mathbb N$. By using the notation $\pFq{3}{2}{}{}{}$ from the theory of generalized hypergeometric functions \cite{Watson1925}, the above series can also be represented as 
	$$
	\sum_{\kk = 0}^{\infty} \frac{  (i-j)_\kk (-s)_\kk ( -i-j-1 )_\kk  }{ (-2s)_\kk (-j)_\kk   } \frac{1}{ \kk ! }=\pFq{3}{2}{i-j,-s,-i-j-1}{ -2s, -j }{1} = \pFq{3}{2}{-n,c,2c+2d+n-1}{ 2c, c+d }{1} 
	$$
	with $n:=j-i \in \mathbb N$, $c:=-s$, and $d:=s-j$. We use Watson's formula \cite{Watson1925} 
	%discovered the following formula 
	 for such hypergeometric series: 
\begin{equation}\label{eq:Watson}
		\pFq{3}{2}{-n,c,2c+2d+n-1}{ 2c, c+d }{1} = 
		\begin{cases}
		\displaystyle
	 \frac{ n! \Gamma( c+\frac12 n ) \Gamma(d+\frac12 n) \Gamma(2c) \Gamma(c+d) }{ (\frac12 n)! \Gamma( c+d+\frac12 n ) \Gamma( 2 c + n ) \Gamma(c) \Gamma(d) },~ & \mbox{if $n$ is even,} 
	 \\
	 0,~ & \mbox{if $n$ is odd.} 
	\end{cases}
\end{equation}
If $n=j-i$ is even, define $m :=  \frac{j-i}2 = \frac12 n \in \mathbb N$. 
	Note that the singularity in \cref{eq:Watson} is removable, 
	%in the sense of limit \zsc{What does it mean by ``in the sense of limit"? Maybe remove it.}, 
	because 
	\begin{align}\label{eq:WKL11}
	&	\frac{ \Gamma( x+m ) } { \Gamma(x) } = \frac{ (x+m-1) \Gamma(x+m-1) }{ \Gamma(x) } = \dots = \prod_{\kk=0}^{m-1} (x+\kk)  \neq 0, \qquad x=c,d,c+d,
	\\ \label{eq:WKL22}
	& \frac{ \Gamma(2c) }{ \Gamma(2c+n) } = \frac{ \Gamma(2c) }{ (2c+n-1) \Gamma(2c+n-1) } = \dots = \prod_{\kk=0}^{n-1} \frac1{ 2c+\kk } 
	=  \frac{ (2s-j+i)! }{ (2s)! } > 0,
	\end{align}
	where the formula $\Gamma(x+1)=x\Gamma(x)$ has been used repeatedly. 
	It follows from \cref{eq:WKL11} that 
	\begin{align}\label{eq:WKL33}
		&\frac{ \Gamma( c+\frac12{n} ) } { \Gamma(c) } = \prod_{\kk=0}^{m-1} ( -s +\kk) = (-1)^m \frac{ s! }{ (s-m)! } = (-1)^m \frac{ s! }{ (s-\frac{j-i}2)! },
		\\ \label{eq:WKL44}
	 & \frac{ \Gamma( d+\frac12{n} ) } { \Gamma(d) } = \prod_{\kk=0}^{m-1} ( s-j +\kk) = ( s-j ) \frac{ ( s-j+m-1 )! }{ (s-j)! } 
	 = (s-j) \frac{ (s-1-\frac{i+j}2)! }{ (s-j)! },
	 \\ \label{eq:WKL55}
	 & \frac{ \Gamma(c+d) }{ \Gamma(c+d+\frac12 n) } %=  \prod_{l=0}^{m-1} (c+d+l)^{-1} 
	 = (-1)^m \prod_{\kk=0}^{m-1} (j-\kk)^{-1} = (-1)^m \frac{ (j-m)! }{ j! } = (-1)^m \frac{ ( \frac{i+j}2 )! }{ j! }.
	\end{align}
	Substituting \cref{eq:WKL22}--\cref{eq:WKL55} into \cref{eq:Watson} and combining \cref{eq:WKL00} with \cref{eq:Watson}, we obtain for $i \equiv j~({\mathrm {mod}}~2)$ that 
	\begin{equation*}
		\sum_{\kk = 0}^{j-i} a_\kk = a_0 \frac{ n! }{ (\frac12 n)! }  \frac{ s! }{ (s-\frac{j-i}2)! }  (s-j) \frac{ (s-1-\frac{i+j}2)! }{ (s-j)! }  \frac{ (2s-j+i)! }{ (2s)! }  \frac{ ( \frac{i+j}2 )! }{ j! } 
		=\frac{(\mys-1-\frac{i+j}{2})!(\frac{i+j}{2})!}{(\mys-\frac{j-i}{2})!(\frac{j-i}{2})!}(s-j),
	\end{equation*}
which completes the proof. 
\end{proof}

%section{Proof of \cref{lem:uhatu}} 

%\section{Proof of \cref{lem:uuhat}}

\section{Proof of \cref{lem:identities}}\label{sec:identities}

\begin{proof}
	By the definition of the Pochhammer symbol, one can deduce that 
	\begin{align*}
		(x+n)! &= x! (x+1) (x+2) \cdots (x+n) = x! (x+1)_n,
		\\
		(x)_n & = \left( \prod\limits_{0\le 2i \le n-1}  (x+2i) \right) \cdot \left( \prod\limits_{0\le 2i+1 \le n-1}  (x+2i+1) \right)
		\\
		& =  x (x+2) \cdots \left( x + 2 \left \lceil{ \frac{n}2 }\right \rceil -2  \right)   \cdot  (x+1) (x+3) \cdots \left( x + 2 \left \lfloor{ \frac{n}2 }\right \rfloor -1  \right)  
		\\
		& =  2^{ \left \lceil{ \frac{n}2 }\right \rceil }  \left( \frac{x}2 \right)_{ \left \lceil{ \frac{n}2 }\right \rceil  }  \cdot   2^{ \left \lfloor{ \frac{n}2 }\right \rfloor }   \left( \frac{x+1}2 \right)_{ \left \lfloor{ \frac{n}2 }\right \rfloor  } 
		%\\& 
		= 
		2^n \left( \frac{x}2 \right)_{ \left \lceil{ \frac{n}2 }\right \rceil  } \left( \frac{x+1}2 \right)_{ \left \lfloor{ \frac{n}2 }\right \rfloor  },
		\\
		\frac{ (x+i)! }{ (x-j)! } &=  (x-j+1)(x-j+2) \cdots (x-1) x \cdot (x+1) \cdots (x+i)
		%\\& 
		= (-1)^j (-x)_j (x+1)_i. 
	\end{align*}
\end{proof}

\section{Proof of \cref{lem:alpha-beta}}\label{sec:alpha-beta}

\begin{proof}
	If $i>j$, then by definition \eqref{eq:Defalpha} we know that $  \myalpha_{2i,2j}^{(\mys)} = \myalpha_{2i-1,2j-1}^{(\mys)} =0 $. 
	On the other hand, when $i>j$, we have $(-j)_i=0$ and $(1-j)_{i-1}=0$, 
	which imply the right-hand sides of \eqref{eq:AB1} and \eqref{eq:AB2} are both zero. Hence the identities \eqref{eq:AB1} and \eqref{eq:AB2} are true for  $i>j$. 
	In the following, we focus on the nontrivial case that $i\le j$. 
	
	{\tt Proof of \eqref{eq:AB1} for $i\le j$.} We observe that 
	\begin{align*}
		\myalpha_{2i,2j}^{(\mys)} & = \frac{ \mys! }{ (2\mys)! } \frac{ 2 \sqrt{4i-1} }{ (2i+2j)! } 
		\frac{ (2\mys+2i-2j)! ( \mys - i -j )! (  i+j )!  } 
		{ (\mys-2j)! (\mys - j+i )! (  j-i )!  } =  \frac{ \mys! }{ (2\mys)! } \frac{ 2 \sqrt{4i-1} }{ (\mys-2j)! }  \Pi_1 \Pi_2
	\end{align*}
	with 
	\begin{align*}
		\Pi_1 &:= \frac{ (\mys-i-j)! }{ (\mys-j+i)! } ( 2 \mys - 2 j + 2i )! 
		%\\& 
		\stackrel{\eqref{ID-1}}{=} \frac{ (\mys-j-i)! }{ (\mys-j+i)! } ( 2\mys - 2j )! (2\mys - 2j +1)_{2i}
		\\
		& \stackrel{\eqref{ID-2}}{=}  \frac{ (\mys-j-i)! }{ (\mys-j+i)! } ( 2\mys - 2j )! 
		2^{2i}  \left( \mys-j+\frac12 \right)_i ( \mys-j+1 )_i
		\\
		& \stackrel{\eqref{ID-3}}{=}  \frac{ 1 }{ (-1)^i (j-\mys)_i (\mys-j+1)_i } ( 2\mys - 2j )! 
		2^{2i}  \left( \mys-j+\frac12 \right)_i ( \mys-j+1 )_i	
		%\\	& 
		= \frac{ (2\mys-2j)! 2^{2i} \left( \mys -j +\frac12 \right)_i }{ (j-\mys)_i (-1)^i },
	\end{align*}
	and 
	\begin{align*}
		\Pi_2 &:= \frac{ (j+i)! }{ (j-i)! } \frac1 { (2i+2j)! } 
		%\\ &  
		\stackrel{\eqref{ID-3}}{=} \frac{ (-1)^i (-j)_i (j+1)_i }{ (2i+2j)! }
		%\\& 
		\stackrel{\eqref{ID-1}}{=} \frac{ (-1)^i (-j)_i (j+1)_i }{ (2j)! (2j+1)_{2i} }
		\\ & 
		\stackrel{\eqref{ID-2}}{=} \frac{ (-1)^i (-j)_i (j+1)_i }{ (2j)! 2^{2i} 
			\left(j + \frac12 \right)_i (j+1)_i }
		%\\&
		= \frac{ (-1)^i (-j)_i }{ 2^{2i} \left( j+\frac12 \right)_i } 
		\frac1{ (2j)! }.
	\end{align*}
	It follows that 
	$$
	\Pi_1  \Pi_2 = \frac{ ( 2\mys -2j )! }{ (2j)! } \frac{ \left( \mys-j+\frac12 \right)_i (-j)_i  }{ (j-\mys)_i \left( j+\frac12 \right)_i }. 
	$$ 
	Therefore, we obtain
	\begin{align*}
		\myalpha_{2i,2j}^{(\mys)} &= 2 \sqrt{4i-1} 	\frac{ \mys! }{ (2\mys)! } 
		\frac{ ( 2\mys -2j )! }{ (2j)! (\mys-2j)! } \frac{ \left( \mys-j+\frac12 \right)_i (-j)_i  }{ (j-\mys)_i \left( j+\frac12 \right)_i } 
		%\\&
		= 2 \sqrt{4i-1}  \mybeta_{2j}^{(\mys)} \frac{ \left( \mys-j+\frac12 \right)_i (-j)_i  }{ (j-\mys)_i \left( j+\frac12 \right)_i },
	\end{align*}
	which yields \eqref{eq:AB1}. 
	
	{\tt Proof of \eqref{eq:AB2} for $i\le j$.}  We observe that 
	\begin{align*}
		\myalpha_{2i-1,2j-1}^{(\mys)} &= \frac{\mys!}{ (2\mys)! } 
		\frac{ 2\sqrt{4i-3} }{ (2i+2j-2)! } 
		\frac{ (2\mys-2j+2i)! }{ (\mys-j+i)! } 
		\frac{ (\mys-i-j+1)! }{ (\mys-2j+1)! } 
		\frac{ (j+i-1)! } { (j-i)! }
		%\\&
		=\frac{ \mys! }{ (2\mys)! } \frac{ 2\sqrt{4i-3} }{ (\mys-2j+1)! } 
		\Pi_3 \Pi_4
	\end{align*}
	with 
	\begin{align*}
		\Pi_3 & := \frac{ (\mys-i-j+1)! }{ (\mys-j+i)! } ( 2\mys-2j+2i )! 
		%\\& 
		\stackrel{\eqref{ID-1}}{=} \frac{ (\mys-i-j+1)! }{ (\mys-j+i)! } ( 2\mys-2j+1 )! (2\mys-2j+2)_{2i-1}
		\\
		& \stackrel{\eqref{ID-2}}{=} \frac{ (\mys-i-j+1)! }{ (\mys-j+i)! } ( 2\mys-2j+1 )! 2^{2i-1} ( \mys-j+1 )_i \left(\mys-j+\frac32 \right)_{i-1}
		\\
		& \stackrel{\eqref{ID-3}}{=} \frac{ ( 2\mys-2j+1 )! }{  (-1)^{i-1} (j-\mys)_{i-1} (\mys-j+1)_i  }  2^{2i-1} ( \mys-j+1 )_i \left( \mys-j+\frac32 \right)_{i-1}
		\\
		& = \frac{ (2\mys-2j+1)! 2^{2i-1} \left( \mys+\frac32 - j \right)_{i-1} }
		{ (-1)^{i-1} (j-\mys)_{i-1}  },
	\end{align*}
	and 
	\begin{align*}
		\Pi_4 & := \frac{ (j+i-1)! }{ (j-i)! } \frac{1}{ (2j+2i-2)! } 
		\stackrel{\eqref{ID-3}}{=} \frac{ (-1)^i (-j)_i (j+1)_{i-1} }{ (2j+2i-2)! }  
		=\frac{ (-1)^{i-1} (1-j)_{i-1} (j)_i }{ (2j+2i-2)! } 
		\\
		&
		\stackrel{\eqref{ID-1}}{=} \frac{ (-1)^{i-1} (1-j)_{i-1} (j)_i }
		{ (2j-1)! (2j)_{2i-1} } 
		\stackrel{\eqref{ID-2}}{=} \frac{ (-1)^{i-1} (1-j)_{i-1} (j)_i }
		{ (2j-1)! 2^{2i-1} (j)_i \left( j+\frac12 \right)_{i-1} } 
		= \frac{ (-1)^{i-1} (1-j)_{i-1}  }{ 2^{2i-1} \left( j+\frac12 \right)_{i-1} } \frac1{ (2j-1)! }.
	\end{align*}
	It follows that 
	\begin{equation*}
		\Pi_3 \Pi_4 = \frac{ (2\mys-2j+1)! }{ (2j-1)! } 
		\frac{ \left( \mys+\frac32-j \right)_{i-1} }{ (j-\mys)_{i-1} } 
		\frac{ (1-j)_{i-1} }{ \left( j+\frac12 \right)_{i-1} }.
	\end{equation*}
	Therefore, we complete the proof by noting  
	\begin{align*}
		\myalpha_{2i-1,2j-1}^{(\mys)} &= 2\sqrt{4i-3} \frac{ \mys! }{ (2\mys)! } \frac{ (2\mys-2j+1)! }{ (2j-1)! (\mys-2j+1)! } 
		\frac{ \left(\mys+\frac32-j \right)_{i-1} (1-j)_{i-1} }
		{ (j-\mys)_{i-1}\left( j+\frac12 \right)_{i-1} } 
		\\
		& = 2\sqrt{4i-3} \mybeta_{2j-1}^{ (\mys) } \frac{ \left( \mys+\frac32-j \right)_{i-1} (1-j)_{i-1} }
		{ (j-\mys)_{i-1}\left( j+\frac12 \right)_{i-1} }. 
	\end{align*}

	%The proof is completed. 
	
\end{proof}

%\section{An example appendix} 
%\lipsum[71]
%\begin{lemma}
%	Test Lemma.
%\end{lemma}
%\newpage

\renewcommand\baselinestretch{0.926}

\bibliographystyle{siamplain}
\bibliography{refs}

\end{document}

%% file: ex_shared.tex
\usepackage{amsfonts}
\usepackage{graphicx,epstopdf}
\usepackage{mathrsfs}
\usepackage{bm}
\usepackage{multirow}
\usepackage{color}
\usepackage{hyperref} 
\usepackage{xcolor}

\usepackage{caption, subcaption}%
\usepackage{enumitem}   

\usepackage{geometry}   
\allowdisplaybreaks[3]
\ifpdf
  \DeclareGraphicsExtensions{.eps,.pdf,.png,.jpg}
\else
  \DeclareGraphicsExtensions{.eps}
\fi

\newsiamremark{remark}{Remark}
\newsiamremark{example}{Example}
\newsiamremark{hypothesis}{Hypothesis}
\crefname{hypothesis}{Hypothesis}{Hypotheses}
\newsiamthm{claim}{Claim}
\newsiamthm{assum}{Assumption}
\newsiamthm{observation}{Observation}
\headers{Energy Laws and Stability of Runge--Kutta Methods}{Zheng Sun, Yuanzhe Wei, Kailiang Wu}

\title{On Energy Laws and Stability of Runge--Kutta Methods for Linear Seminegative Problems} % via Energy Method Energy Equality for Implicit Runge--Kutta Methods for Seminegative Definite Problems  and Norm Equality
\author{
Zheng Sun\thanks{Department of Mathematics, The University of Alabama,
	Tuscaloosa, AL 35487, USA
		({\tt zsun30@ua.edu}).}
		\and
	Yuanzhe Wei\thanks{Department of Mathematics, Southern University of Science and Technology, Shenzhen, Guangdong 518055, China ({\tt weiyz2019@mail.sustech.edu.cn}).}
	\and 
	Kailiang Wu\thanks{Corresponding author. Department of Mathematics, Southern University of Science and Technology, and National Center for Applied Mathematics Shenzhen (NCAMS), Shenzhen, Guangdong 518055, China ({\tt wukl@sustech.edu.cn}). Tel: +86-755-88010575. The work of K.~Wu is supported in part by NSFC grant 12171227.}
}

\usepackage{amsopn}
\DeclareMathOperator{\diag}{diag}
\DeclareMathAlphabet\mathbfcal{OMS}{cmsy}{b}{n}

\usepackage{stmaryrd}

\newcommand{\myalpha}{\nu}
\newcommand{\mybeta}{\theta}

\newcommand{\mys}{s}
\newcommand{\myN}{N}
\newcommand{\kk}{\ell}
\newcommand{\dt}{\frac{d}{dt}}

\newcommand{\nm}[1]{\left\| #1 \right\|}

\newcommand{\qnm}[1]{\left\llbracket #1 \right\rrbracket}

\newcommand{\qip}[1]{\left[ #1 \right]}
\newcommand{\ip}[1]{\left\langle #1 \right\rangle}

\newcommand{\bfmyLambda}{\mathbf{D}}
\newcommand{\bfhatmyLambda}{\mathbf{\widehat{D}}}
\newcommand{\bftildemyLambda}{\mathbf{\tilde{D}}}
\newcommand{\mylambda}{d}
\newcommand{\mycfl}{\lambda}
\newcommand{\quand}{\quad\text{and}\quad}
\usepackage{algpseudocode}
\algnewcommand\algorithmicswitch{\textbf{switch}}
\algnewcommand\algorithmiccase{\textbf{case}}
\algnewcommand\algorithmicdefault{\textbf{default}}
\algnewcommand\Assert[1]{\State \algorithmicassert(#1)}%
\algdef{SE}[SWITCH]{Switch}{EndSwitch}[1]{\algorithmicswitch\ #1\ \algorithmicdo}{\algorithmicend\ \algorithmicswitch}%
\algdef{SE}[CASE]{Case}{EndCase}[1]{\algorithmiccase\ #1}{\algorithmicend\ \algorithmiccase}%
\algdef{SE}[CASE]{Default}{EndDefault}{\algorithmicdefault}{\algorithmicend\ \algorithmic}%
\algtext*{EndSwitch}%
\algtext*{EndCase}%
\algtext*{EndDefault}%

\newmuskip\pFqmuskip
\newcommand*\pFq[6][8]{%
	\begingroup % only local assignments
	\pFqmuskip=#1mu\relax
	\mathchardef\normalcomma=\mathcode`,
	\mathcode`\,=\string"8000
	\begingroup\lccode`\~=`\,
	\lowercase{\endgroup\let~}\pFqcomma
	{}_{#2}F_{#3}{\left(  \genfrac..{0pt}{}{#4}{#5} \right)}%
	\endgroup
}
\newcommand{\pFqcomma}{{\normalcomma}\mskip\pFqmuskip}